\documentclass[11pt]{amsart}
\usepackage[margin=1.1in]{geometry}
\usepackage{amsfonts,amsthm,amsmath,amssymb}
\usepackage{amscd}
\usepackage{graphicx}

\usepackage{hyperref} 
\hypersetup{colorlinks=true, linkcolor=blue, citecolor=blue, filecolor=blue, urlcolor=blue}
\usepackage[comma, sort&compress]{natbib}

\newcommand{\one}{{\bf 1}}

\newcommand{\reals}{{\mathbb R}}
\newcommand{\bbr}{\reals}
\newcommand{\vep}{\varepsilon}

\newcommand{\bbn}{\protect{\mathbb N}}
\newcommand{\bbc}{\protect{\mathbb C}}

\newcommand{\bbg}{\protect{\mathbb G}}

\newcommand{\A}{{\mathcal A}}

\newcommand{\R}{{\mathbb R}}

\newtheorem{theorem}{Theorem}[section]

\newtheorem{prop}{Proposition}[section]
\newtheorem{fact}{Fact}[section]
\newtheorem{lemma}{Lemma}[section]
\newtheorem*{defn}{Definition}

\newtheorem{remark}{Remark}[section]

\def\Cov{{\rm Cov}}

\def\Var{{\rm Var}}
\def\E{{\rm E}}
\def\esd{{\rm ESD} }
\def\eesd{{\rm EESD}}

\def\supp{{\rm Supp}}

\DeclareMathOperator{\Tr}{Tr} 

\DeclareMathOperator{\diag}{Diag}

\numberwithin{equation}{section}

\begin{document}

\title[Inhomogeneous Erd\H{o}s-R\'enyi random graphs]{Spectra of Adjacency and Laplacian\\ 
Matrices of Inhomogeneous\\ Erd\H{o}s-R\'enyi Random Graphs}

\author[A. Chakrabarty]{Arijit Chakrabarty}
\author[R.~S.~Hazra]{Rajat Subhra Hazra}
\address{Indian Statistical Institute, 203 B.T. Road, Kolkata 700108, India}
\email{arijit.isi@gmail.com}
\email{rajatmaths@gmail.com}
\author[F.\ den Hollander]{Frank den Hollander}
\author[M.\ Sfragara]{Matteo Sfragara} 
\address{Mathematisch instituut, Universiteit Leiden, The Netherlands}
\email{denholla@math.leidenuniv.nl}
\email{m.sfragara@math.leidenuniv.nl}

\begin{abstract}
The present paper considers inhomogeneous Erd\H{o}s-R\'enyi random graphs $\bbg_N$ on 
$N$ vertices in the non-sparse non-dense regime. The edge between the pair of vertices $\{i,j\}$ is 
retained with probability $\vep_N\,f(\tfrac{i}{N},\tfrac{j}{N})$, $1 \leq i \neq j \leq N$, independently 
of other edges, where $f\colon\,[0,1] \times [0,1] \to [0,\infty)$ is a continuous function such that 
$f(x,y)=f(y,x)$ for all $x,y \in [0,1]$.  We study the empirical distribution of both the adjacency 
matrix $A_N$ and the Laplacian matrix $\Delta_N$ associated with $\bbg_N$, in the limit as 
$N \to \infty$ when $\lim_{N\to\infty} \vep_N = 0$ and $\lim_{N\to\infty} N\vep_N = \infty$. In 
particular, we show that the empirical spectral distributions of $A_N$ and $\Delta_N$, after 
appropriate scaling and centering, converge to deterministic limits weakly in probability. For 
the special case where $f(x,y) = r(x)r(y)$ with $r\colon\,[0,1] \to [0,\infty)$ a continuous function, 
we give an explicit characterization of the limiting distributions. Furthermore, we apply our 
results to constrained random graphs, Chung-Lu random graphs and social networks.
\end{abstract}


\maketitle


\section{Introduction and main results}
\label{sec:intro}

Spectra of random matrices have been analyzed for almost a century. In recent 
years, many interesting results have been derived for random matrices associated 
with \emph{random graphs}, like the adjacency matrix and the Laplacian matrix 
(\cite{bauer2001random, farkas2001spectra, tran2013sparse, dumitriu2012sparse, 
bordenave2010resolvent, ding2010spectral, khorunzhy2004eigenvalue, bhamidi2012spectra, 
jiang2012low, jiang2012empirical, lee:schnelli}).

The focus of the present paper is on \emph{inhomogeneous} Erd\H{o}s-R\'enyi random 
graphs, which are rooted in the theory of complex networks. We consider the regime 
where the degrees of the vertices diverge sublinearly with the size of the graph. In this 
regime, which is neither sparse nor dense, we identify the scaling limit of the empirical 
spectral distribution, both for the adjacency matrix and the Laplacian matrix. For the 
special case where the connection probabilities have a \emph{product structure}, we 
obtain an explicit description of the scaling limit of the empirical spectral distribution in 
terms of objects that are rooted in free probability. It is known that in the absence of 
inhomogeneity, i.e., for standard Erd\H{o}s-R\'enyi random graphs, in the sparse regime 
the empirical spectral distributions of the adjacency matrix and the Laplacian matrix 
converge (after appropriate scaling and centering) to a semicircle law, respectively, 
a free additive convolution of a Gaussian and a semicircle law (see, for example, 
\cite{ding2010spectral,jiang2012empirical,bryc:dembo:jiang:2006}). Our results 
extend the latter to the inhomogeneous setting. 

There are some recent results on the largest eigenvalue of sparse inhomogeneous Erd\H{o}s-R\'enyi 
random graphs (\cite{benaych2017largest, benaych:2017}), and also on the empirical spectral 
distribution of adjacency matrices via the theory of graphons (\cite{zhu:2018}). Inhomogeneous 
Erd\H{o}s-R\'enyi random graphs with a product structure for the connection probabilities 
arise naturally in different contexts. In \cite{dembo:eyal}, the latter has been shown to play a 
crucial role in the identification of the limiting spectral distribution of the adjacency matrix 
of the configuration model. Our methodology allows us to look at some important applications. 
For instance, a \emph{Chung-Lu type random graph} is used to model \emph{sociability patterns} 
in networks. We show how to use the rescaled empirical spectral distribution and free probability 
to \emph{statistically recover the underlying sociability distribution}. Another important application 
is constrained random graphs. Given a sequence of positive integers, among the probability 
distributions for which the sequence of \emph{average degrees} matches the given sequence, 
called the \emph{soft configuration model}, the one that maximizes the entropy is the 
\emph{canonical Gibbs measure}. It is known that, under a sparsity condition, the connection 
probabilities arising out of the canonical Gibbs measure asymptotically have a product 
structure (\cite{squartini:demol:denhollander:garlaschelli:2015}). We show that our results on 
the adjacency matrix can be easily extended to cover such situations. The spectrum of the 
Laplacian of a random graph is well known to be connected to properties of the \emph{random 
walk} on the random graph, algebraic connectivity, and Kirchhoff's law, among others. The 
explicit bearing of the spectral distribution of the Laplacian on the corresponding graph are left 
for future research, for which our results may serve as a starting point.
          
The paper is organized as follows. The setting is defined in Section~\ref{sec:intro}, and four
scaling theorems are stated: Theorems~\ref{t.adj.gen}--\ref{t.random}. A number of technical 
lemmas are stated and proved in Section~\ref{sec:preps}. These serve as preparation for the 
proof of Theorems~\ref{t.adj.gen}--\ref{t.mult}, which is given in Section~\ref{sec:proofs}. 
Theorem \ref{t.random}, which is a randomized version of Theorem~\ref{t.adj.gen}, is proved in 
Section~\ref{sec:proofrandom}. In Section~\ref{sec:appl}, three applications are discussed,
organized into three propositions. Appendix~\ref{sec:facts} collects a few basic facts that are 
needed along the way.


\subsection{Setting}
\label{sec:setting}

Let $f\colon\,[0,1]\times[0,1]\to[0,\infty)$ be a continuous function, satisfying
\[
f(x,y)=f(y,x) \qquad \forall\,x,y \in[0,1]\,.
\]
A sequence of positive real numbers $(\vep_N\colon\,N \geq 1)$ is fixed that satisfies
\begin{equation}
\label{eq.conditions}
\lim_{N\to\infty}\vep_N=0\,, \qquad \lim_{N\to\infty}N\vep_N=\infty\,.
\end{equation}
Consider the random graph $\bbg_N$ on vertices $\{1,\ldots,N\}$ where, for each $(i,j)$ with 
$1\le i<j\le N$, an edge is present between vertices $i$ and $j$ with probability $\vep_Nf(\tfrac iN,
\tfrac jN)$, independently of other pairs of vertices. In particular, $\bbg_N$ is an undirected 
graph with no self loops and no multiple edges. Boundedness of $f$ ensures that $\vep_N
f(\tfrac iN,\tfrac jN) \leq 1$ for all $1\le i < j \le N$ when $N$ is large enough. If $f \equiv c$ 
with $c$ a constant, then $\bbg_N$ is the Erd\H{o}s-R\'enyi graph with edge retention probability
$\vep_N c$. For general $f$, $\bbg_N$ can be thought of as an \emph{inhomogeneous} version 
of the Erd\H{o}s-R\'enyi graph. 

The adjacency matrix of $\bbg_N$ is denoted by $A_N$. Clearly, $A_N$ is a symmetric random 
matrix whose diagonal entries are zero and whose upper triangular entries are independent
Bernoulli random variables, i.e.,  
\[
A_N(i,j) \triangleq \text{BER}\left(\vep_Nf\left(\tfrac iN,\tfrac jN\right)\right)\,,
\qquad 1\le i\neq j\le N\,.
\]
Write $P$ to denote the law of $A_N$. 


\subsection{Scaling}

Our first theorem states the existence of the limiting spectral distribution of $A_N$ after suitable scaling.
Here, and elsewhere in the paper, $\esd$ is the abbreviation for \emph{empirical spectral distribution}: 
the probability measure that puts mass $1/N$ at every eigenvalue, respecting its algebraic multiplicity.

\begin{theorem}
\label{t.adj.gen}
There exists a compactly supported and symmetric probability measure $\mu$ on $\bbr$ such that
\begin{equation}
\label{t.adj.gen.claim1}\lim_{N\to\infty} \esd\left((N\vep_N)^{-1/2}A_N\right) 
= \mu \quad \text{ weakly in probability}\,. 
\end{equation}
Furthermore, if 
\[
\min_{0\le x,y\le1}f(x,y)>0\,,
\]
then $\mu$ is absolutely continuous with respect to Lebesgue measure.
\end{theorem}

The Laplacian of $\bbg_N$ is the $N \times N$ matrix $\Delta_N$ defined as
\[
\Delta_N(i,j)=
\begin{cases}
-\sum_{k=1}^NA_N(i,k), &i=j\,,\\
A_N(i,j),&i \neq j\,.
\end{cases}
\]
Our second theorem is the analogue of Theorem \ref{t.adj.gen} with $A_N$ replaced by $\Delta_N$.

\begin{theorem}
\label{t.lap.gen}
There exists a symmetric probability measure $\nu$ on $\bbr$ such that
\[
\lim_{N\to\infty} \esd\left((N\vep_N)^{-1/2}\bigl(\Delta_N-D_N\bigr)\right) 
= \nu \quad \text{ weakly in probability}\,, 
\]
where 
\begin{equation}
\label{DNdef}
D_N=\diag\left(\E\bigl(\Delta_N(1,1)\bigr),\ldots,\E\bigl(\Delta_N(N,N)\bigr)\right)\,.
\end{equation}
Furthermore, if
\[
f \not\equiv 0,
\]
then the support of $\nu$ is unbounded.
\end{theorem}

\begin{remark}
{\rm The $\esd$ of a random matrix is a random probability measure. Note that $\mu$ and 
$\nu$ are both deterministic, i.e., a law of large numbers is in force.

\medskip\noindent
(1) Theorems \ref{t.adj.gen} and \ref{t.lap.gen} are existential, in the sense that 
explicit descriptions of $\mu$ and $\nu$ are missing. We have some control on the 
Stieltjes transform of $\mu$. In the proof of Theorem~\ref{t.adj.gen} (in Lemma~\ref{l.minor}) 
we will see that the ESD of $(N\vep_N)^{-1/2} A_N$ has the same limit as the ESD of 
\[ 
\bar A_N(i,j)= \sqrt{\frac{1}{N}f\left( \frac{i}{N}, \frac{j}{N}\right)}\, G_{i\wedge j, i\vee j}
\]
with $(G_{i,j}: \, 1\le i\le j)$ a family of i.i.d.\ standard Gaussian random variables. Such random 
matrices are known in the literature as Wigner matrices with a variance profile (see, for example, 
\cite{ajanki2017,hachem:2007,shlyakhtenko:1996,chakrabarty:2017}). The LSD of $\bar A_N$ matches with the LSD of certain symmetric random matrices 
with dependent entries; see \cite{chakrabarty:hazra:sarkar:2016} for details. It turns out that, by 
using the combinatorics of non-crossing partitions, we can derive a recursive equation for the 
Stieltjes transform of $\mu$, i.e.,
\[
G_{\mu}(z)= \int_{\mathbb R} \frac{1}{z-x}\, \mu(dx)\,, \quad z\in \mathbb C\setminus \mathbb R\,.
\] 
It turns out that 
\[
G_\mu(z)= \int_0^1\mathcal H(z,x)\, dx\,, 
\]
where $\mathcal H(z,x)$, $x \in [0,1]$, is the unique analytic solution of the integral equation 
\[
z\mathcal H(z,x)=1+\mathcal H(z,x) \int_0^1 \mathcal H(z,y) f(x,y)\,dy\,, \,\qquad x\in [0,1]\,.
\]
The form of $\mathcal H(z,x)$ can also be expressed in terms of non-crossing partitions and the 
function $f(x,y)$ (see \cite[Section 4.1]{chakrabarty:hazra:sarkar:2015} for details). We mention 
that the above measure is similar to the limiting measure in \cite[Theorem 3.4]{zhu:2018}. It is 
shown in \cite{zhu:2018} that a graphon sequence $W_N$ can be associated with a Wigner 
matrix with a variance profile $(s_{i,j}: 1\le i, j\le N)$. If the sequence of graphons $W_N$
converges in the cut norm to $W$ with $W(x,y)=f(x,y)$, then the limiting measure matches 
with $\mu$.  

\medskip\noindent
(2) The description of $\nu$ through its Stieltjes transform is hard to obtain, although, just 
like before, the ESD of $(N\vep_N)^{-1/2}(\Delta_N-D_N)$ turns out to same as that of 
\[
\tilde\Delta_N=\bar A_N+Y_N
\] 
with
\[
Y_N(i,i)=Z_i\,\sqrt{\frac1N\sum_{1\le j\le N,\,j\neq i}f\left(\frac iN,\frac jN\right)}\,,
\qquad 1\le i\le N\,,
\]
where $(Z_i\colon\,i \ge 1)$ is a family of i.i.d.\ standard normal random variables, independent 
of $(G_{i,j}\colon\,1\le i \le j)$. Suppose that $Y_N$ is a deterministic diagonal matrix, embedded 
in $L^\infty[0,1]$ (as a step function). For the case where this function converges to a function 
$h$ in the $\|\cdot\|_\infty$ norm, the limiting spectral distribution of $\bar A_N+Y_N$ was 
studied in \cite{shlyakhtenko:1996} (see also \cite[Theorem 22.7.2]{Speicher:survey}). In our 
case, due to the presence in $Y_N$ of Gaussian random variables (which have unbounded 
support) and the fact that the spectral norm of $Y_N$ tends to infinity as $N\to\infty$, the existing 
results cannot be applied. One of the major contributions of our paper is to overcome this hurdle. 
Also, our proofs ensure that $\nu$ has a finite moment generating function (see \eqref{t.lap.gen.eq7}
below) and unbounded support.}
\end{remark}


\subsection{Multiplicative structure}

Our next theorem identifies $\mu$ and $\nu$ under the additional assumption that $f$ has a 
\emph{multiplicative structure}, i.e.,
\begin{equation}
\label{eq.mult}
f(x,y)=r(x)r(y)\,, \qquad x,y \in [0,1]\,,
\end{equation}
for some continuous function $r\colon\, [0,1] \to [0,\infty)$. The statement is based on the theory 
of (possibly unbounded) self-adjoint operators affiliated with a $W^*$-probability space. A few 
relevant definitions are given below. For details the reader is referred to 
\cite[Section 5.2.3]{anderson:guionnet:zeitouni:2010}.

\begin{defn}
A $C^*$-algebra $\A\subset B({\mathcal H})$, with $\mathcal H$ a Hilbert space, is a 
\emph{$W^*$-algebra} when $\A$ is closed under the weak operator topology. If, in addition, 
$\tau$ is a state such that there exists a unit vector $\xi\in\mathcal H$ satisfying
\[
\tau(a)=\langle a\xi,\xi\rangle \qquad \forall\,a\in\mathcal H\,,
\]
then $(\A,\tau)$ is a \emph{$W^*$-probability space}. In that case a densely defined 
self-adjoint (possibly unbounded) operator $T$ on $\mathcal H$ is said to be \emph{affiliated 
with $\A$} if $h(T)\in\A$ for any bounded measurable function $h$ defined on the spectrum 
of $T$, where $h(T)$ is defined by the spectral theorem. Finally, for an affiliated operator 
$T$, its \emph{law} ${\mathcal L}(T)$ is the unique probability measure on $\bbr$ satisfying
\[
\tau\left(h(T)\right)=\int_\R h(x)\bigl({\mathcal L}(T)\bigr)(dx)
\]
for every bounded measurable $h\colon\,\bbr\to\bbr$.
\end{defn}

The distribution of a single self-adjoint operator is defined above. For two or more self-adjoint 
operators $T_1,\ldots,T_n$, a description of their \emph{joint distribution} is a specification of
\[
\tau\big(h_1\left(T_{i_1}\right) \cdots h_k\left(T_{i_k}\right)\big)\,,
\]
for all $k\ge1$, all $i_1,\ldots,i_k\in\{1,\ldots,n\}$, and all bounded measurable functions $h_1,
\ldots,h_k$ from $\bbr$ to itself. Once the above is specified, it is immediate that ${\mathcal L}
(p(T_1,\ldots,T_k))$ can be calculated for any polynomial $p$ in $k$ variables such that 
$p(T_1,\ldots,T_k)$ is self-adjoint. 

\begin{defn}
Let $(\A,\tau)$ be a $W^*$-probability space and $a_1,a_2\in\A$. Then $a_1$ and $a_2$ are 
\emph{freely independent} if 
\[
\tau\left(p_1(a_{i_1}) \cdots p_n(a_{i_n})\right)=0\,,
\]
for all $n\ge1$, all $i_1,\ldots,i_n\in\{1,2\}$ with $i_j\neq i_{j+1}$, $j=1,\ldots,n-1$, and 
all polynomials $p_1,\ldots,p_n$ in one variable satisfying
\[
\tau\left(p_j(a_{i_j})\right)=0,\qquad j=1,\ldots,n\,.
\]
For (possibly unbounded) operators $a_1,\ldots,a_k$ and $b_1,\ldots,b_m$ affiliated with $\A$, 
the collections $(a_1,\ldots,a_k)$ and $(b_1,\ldots,b_m)$ are \emph{freely independent} if and
only if
\[
p\left(h_1(a_1),\ldots,h_k(a_k)\right)\,\text{ and } \, q\left(g_1(b_1),\ldots,g_m(b_m)\right)\,,
\]
are freely independent for all bounded measurable $h_1,\ldots,h_k$ and $g_1,\ldots,g_m$, and 
all polynomials $p$ and $q$ in $k$ and $m$ non-commutative variables, respectively. It is immediate 
that the two operators in the above display are bounded, and hence belong to $\A$.
\end{defn}

We are now in a position to state our third and last theorem.

\begin{theorem}
\label{t.mult}
If $f$ is as in \eqref{eq.mult}, then
\begin{equation}
\label{t.mult.eq1}\mu={\mathcal L}\left(r^{1/2}(T_u)T_sr^{1/2}(T_u)\right)\,,
\end{equation}
and
\begin{equation}
\label{t.mult.eq2}\nu={\mathcal L}\left(r^{1/2}(T_u)T_sr^{1/2}(T_u)
+\alpha r^{1/4}(T_u)T_gr^{1/4}(T_u)\right)\,,
\end{equation}
where
\[
\alpha=\left(\int_0^1r(x)\,dx\right)^{1/2}\,.
\]
Here, $T_g$ and $T_u$  are commuting self-adjoint operators affiliated with a $W^*$-probability 
space $(\A,\tau)$ such that, for bounded measurable functions $h_1,h_2$ from $\bbr$ to itself,
\begin{equation}
\label{t.mult.eq3}
\tau\left(h_1(T_g)h_2(T_u)\right) = \left(\int_{-\infty}^\infty h_1(x)\phi(x)\,dx \right)
\left(\int_0^1h_2(u)\,du\right)\,,
\end{equation}
with $\phi$ the standard normal density. Furthermore, $T_s$ has a standard semicircle distribution 
and is freely independent of $(T_g,T_u)$. 
\end{theorem}

\begin{remark}
{\rm The right-hand side of \eqref{t.mult.eq1} is the same as the free multiplicative convolution of the 
standard semicircle law and the law of $r(U)$, where $U$ is a standard uniform random variable.}
\end{remark}

\begin{remark}
{\rm The fact that $T_g$ and $T_u$ commute, together with \eqref{t.mult.eq3}, specifies their joint 
distribution. In fact, they are standard normal and standard uniform, respectively, independently of 
each other in the \emph{classical sense}. Free independence of $T_s$ and $(T_g,T_u)$, plus the 
fact that the former follows the standard semicircle law, specifies the joint distribution of $T_s,T_g,T_u$.}
\end{remark}

\begin{remark}
{\rm In order to admit the unbounded operator $T_g$, a $W^*$-probability space is needed. If all 
the operators would have been bounded, then a $C^*$-probability space would have sufficed.}
\end{remark}


\subsection{Randomization}
\label{subsec:random}

Theorem \ref{t.adj.gen} can be generalized to the situation where the function $f$ is random.  
Such a randomization helps us to address the applications listed in Section~\ref{sec:appl}. 
Suppose that $(\vep_N\colon\,N\ge 1)$ is a sequence of positive numbers satisfying 
\eqref{eq.conditions}. Suppose further that, for every $N \ge 1$, $(R_{Ni}\colon\,1\le i\le N)$ 
is a collection of non-negative random variables such that there is a deterministic $C<\infty$ 
for which
\begin{equation}
\label{random.eq1}
\sup_{N\ge1}\max_{1\le i\le N}R_{Ni}\le C\text{ a.s.}
\end{equation}
In addition, suppose that there is a probability measure $\mu_r$ on $\bbr$ such that
\begin{equation}
\label{random.eq.weak}
\lim_{N\to\infty}\frac1N\sum_{i=1}^N\delta_{R_{Ni}}=\mu_r\text{ weakly a.s.}
\end{equation}
The non-negativity of $R_{Ni}$ and \eqref{random.eq1} ensure that $\mu_r$ is concentrated 
on $[0,C]$. Furthermore, the first line of \eqref{eq.conditions} ensures that the additional
assumption
\begin{equation}
\label{random.new}
\sup_{N\ge1}\vep_N\le\frac1C
\end{equation}
entails no loss of generality.

For fixed $N$ and conditional on $(R_{N1},\ldots,R_{NN})$, the random graph $\bbg_N$ is 
constructed as before, except that there is an edge between $i$ and $j$ with probability 
$\vep_NR_{Ni}R_{Nj}$, which is at most $1$ by \eqref{random.new} for all $1\le i<j\le N$. 
In other words, $\bbg_N$ has two levels of randomness: one in the choice of $(R_{N1},\ldots,
R_{NN})$ and one in the choice of the set of edges. Once again, $A_N$ is the adjacency 
matrix of $\bbg_N$. The following is a \emph{randomized} version of Theorem \ref{t.adj.gen}. 

\begin{theorem}\label{t.random}
Under the assumptions \eqref{eq.conditions} and \eqref{random.eq1}--\eqref{random.eq.weak}, 
\[
\lim_{N\to\infty}\esd\left((N\vep_N)^{-1/2}A_N\right)=\mu_r\boxtimes\mu_s \quad\text{ weakly in probability}\,,
\]
where $\mu_s$ is the standard semicircle law. 
\end{theorem}


\section{Preparatory approximations}
\label{sec:preps}

The proofs of Theorems \ref{t.adj.gen}--\ref{t.mult} in Section \ref{sec:proofs} rely on several
preparatory approximations, which we organize in Lemmas~\ref{l.mc}--\ref{l.diag} below. Along
the way we need several basic facts, which we collect in Appendix~\ref{sec:facts}.
 

\subsection{Centering}
 
The first approximation is that the mean of each off-diagonal entry of $A_N$ and $\Delta_N$ 
can be subtracted, with negligible perturbation in the respective empirical spectral distributions.

\begin{lemma}
\label{l.mc}
Let $A^0_N$ and $\Delta^0_N$ be $N\times N$ matrices defined by
\begin{eqnarray*}
A^0_N(i,j) &=& (N\vep_N)^{-1/2}\left[A_N(i,j)-\E\left(A_N(i,j)\right)\right]\,,\\
\Delta^0_N(i,j) &=& (N\vep_N)^{-1/2}\left[\Delta_N(i,j)-\E\left(\Delta_N(i,j)\right)\right]\,,
\end{eqnarray*}
for all $1\le i,j\le N$. Then
\[
\begin{aligned}
&\lim_{N\to\infty} L\left(\esd(A^0_N),\esd((N\vep_N)^{-1/2}A_N)\right) = 0 
\quad \text{in probability}\,,\\
&\lim_{N\to\infty} L\left(\esd(\Delta^0_N),\esd((N\vep_N)^{-1/2}(\Delta_N-D_N))\right) = 0 
\quad \text{in probability}\,,
\end{aligned}
\]
where $L(\eta_1,\eta_2)$ denotes the L\'evy distance between the probability measures $\eta_1$ 
and $\eta_2$, and $D_N$ is the diagonal matrix defined in \eqref{DNdef}.
\end{lemma}

\begin{proof}
An appeal to Fact \ref{fact:HW} shows that 
\begin{eqnarray*}
&&L^3\left(\esd(A^0_N),\esd((N\vep_N)^{-1/2}A_N)\right)\\
&&\qquad \le \frac1{N^2\vep_N}\sum_{i=1}^N\sum_{j=1}^N\E^2\left(A_N(i,j)\right)\\
&&\qquad = \frac1{N^2\vep_N}\sum_{i=1}^N\sum_{j=1,\neq i}^N
\vep_N^2f^2\left(\frac iN,\frac jN\right)\\
&&\qquad = [1+o(1)]\,\vep_N\int_0^1\int_0^1f^2(x,y)\,dx\,dy\,, \qquad N\to\infty\,.
\end{eqnarray*}
The first claim follows by recalling that $\vep_N \downarrow 0$. The proof the second claim is verbatim 
the same.
\end{proof}


\subsection{Gaussianisation}

One of the crucial steps in studying the scaling properties of $\esd$ is to replace each entry by a 
Gaussian random variable.

\begin{lemma}
\label{l.gauss}
Let $(G_{i,j}\colon\,1\le i\le j)$ be a family of i.i.d.\ standard Gaussian random variables. Define 
$N \times N$ matrices $A_N^g$ and $\Delta_N^g$ by
\begin{eqnarray}
\label{l.gauss.eq1}
\qquad A_N^g(i,j) &=&
\begin{cases}
\sqrt{\frac1Nf\left(\frac iN,\frac jN\right)\left(1-\vep_Nf\left(\frac iN,\frac jN\right)\right)}
G_{i\wedge j,i\vee j},&i\neq j\,,\\
0,&i=j\,,
\end{cases}\\
\label{l.gauss.eq2}
\qquad \Delta_N^g(i,j)
&=&\begin{cases}
A_N^g(i,j), &i\neq j\,,\\
-\sum_{ k\in\{1,\ldots,N\}\setminus\{i\} }A_N^g(i,k), &i=j\,.
\end{cases}
\end{eqnarray}
Fix $z\in\bbc\setminus\bbr$ and a three times continuously differentiable function $h\colon\,
\bbr\to\bbr$ such that
\[
\max_{0\le j\le3}\,\sup_{x\in\bbr} |h^{(j)}(x)|<\infty\,.
\]
For an $N\times N$ real symmetric matrix $M$, define
\[
H_N(M)=\frac1N\Tr\left((M-zI_N)^{-1}\right)\,,
\]
where $I_N$ is the identity matrix of order $N$. Then
\begin{eqnarray}
\label{l.gauss.claim1}
\lim_{N\to\infty}
\E\left[h\left(\Re H_N(A_N^g)\right)-h\left(\Re H_N(A_N^0)\right)\right] &=& 0\,,\\
\label{l.gauss.claim2}
\lim_{N\to\infty}
\E\left[h\left(\Im H_N(A_N^g)\right)-h\left(\Im H_N(A_N^0)\right)\right] &=& 0\,,
\end{eqnarray}
and
\begin{eqnarray}
\label{l.gauss.claim3}
\lim_{N\to\infty}
\E\left[h\left(\Re H_N(\Delta_N^g)\right)-h\left(\Re H_N(\Delta_N^0)\right)\right] &=& 0\,,\\
\label{l.gauss.claim4}
\lim_{N\to\infty}
\E\left[h\left(\Im H_N(\Delta_N^g)\right)-h\left(\Im H_N(\Delta_N^0)\right)\right] &=& 0\,,
\end{eqnarray}
where $\Re$ and $\Im$ denote the real and the imaginary part of a complex number, 
respectively.
\end{lemma}

\begin{proof}
We only prove \eqref{l.gauss.claim3}. The proofs of the other claims are similar. We use ideas 
from \cite{chatterjee:2005}. Let $z=u+iv\in \mathbb C^{+}$ and $n=N(N-1)/2$. Define $\phi\colon\,
\bbr^n\to \mathbb C$ as
\begin{equation}
\phi(x)= H_N(\Delta(x))
\end{equation}
where $\Delta(x)$ is the $N \times N$ symmetric Laplacian matrix given by
\[
\Delta(x)(i,j)= \begin{cases}
-\sum_{k=1, k\neq i}^N x_{i,k}  & i=j\\
x_{i\wedge j,i\vee j} &  i\neq j .
\end{cases}
\]
Note that $\partial \Delta(x)/\partial x_{ij}$ is the $N \times N$ matrix that has $-1$ at the $i$-th and 
$j$-th diagonal and $1$ at $(i,j)$-th and $(j,i)$-th entry. The following identities were derived in 
\cite[Section 2]{chatterjee:2005}:
\begin{align}
\nonumber
\frac{\partial \phi}{\partial x_{i,j}}
&= -N^{-1} \Tr\left( \frac{\partial \Delta}{\partial x_{i,j}} K^2\right),\\
\frac{\partial^2 \phi}{\partial x_{i,j}^2}
&= 2 N^{-1} \Tr\left(  \frac{\partial \Delta}{\partial x_{i,j}} 
K\frac{\partial \Delta}{\partial x_{i,j}} K^2\right),\\ \nonumber
\frac{\partial^3 \phi}{\partial x_{i,j}^3}
&= -6 N^{-1} \Tr\left(\frac{\partial \Delta}{\partial x_{i,j}} K\frac{\partial \Delta}
{\partial x_{i,j}} K\frac{\partial \Delta}{\partial x_{i,j}} K^2\right),
\end{align}
where $K(x)= (\Delta(x)-zI)^{-1}$. Now using these identities we get
\[ 
\Big{\|}\frac{\partial \phi}{\partial x_{ij}}\Big{\|}_{\infty} \le \frac{4}{|\Im z|^2}\frac1N,\,\, 
\Big{\|}\frac{\partial^2 \phi}{\partial x_{ij}^2}\Big{\|}_{\infty} \le \frac{8}{|\Im z|^3} 
\frac1N,\,\,\,  \Big{\|}\frac{\partial^3 \phi}{\partial x_{ij}^3}\Big{\|}_{\infty} \le 
\frac{48}{|\Im z|^4} \frac1N\,.
\]
If we define 
\begin{align*}
\lambda_2(\phi)&=\sup\left\{\Big{\|} \frac{\partial \phi}{\partial x_{i,j}}\Big{\|}_\infty^2, \,\,\,
\Big{\|} \frac{\partial^2 \phi}{\partial x_{i,j}^2}\Big{\|}_{\infty}\right\}\,,\\
\lambda_3(\phi)& = \sup\left\{\Big{\|} \frac{\partial \phi}{\partial x_{i,j}}\Big{\|}_\infty^3,
\Big{\|} \frac{\partial^2 \phi}{\partial x_{i,j}^2}\Big{\|}_{\infty}^2, \Big{\|}\frac{\partial^3 \phi}
{\partial x_{i,j}^3}\Big{\|}_\infty\right\}\,,
\end{align*}
then there exists constants $C_2$ and $C_3$ depending on $\Im z$ such that $\lambda_2(\phi)
\le C_2 N^{-1}$ and $\lambda_3(\phi)\le C_3N^{-1}$. Hence, using $\lambda_r(\Re \phi)\le 
\lambda_r(\phi)$ and 
\begin{equation}
\label{UVdef2}
U=\Re \left(H_N(\Delta_N^0)\right), \qquad V= \Im\left( H_N(\Delta_N^g)\right)\,,
\end{equation} 
we have from \cite[Theorem 1.1]{chatterjee:2005} 
\begin{align}
&\left|\E[h(U)]- \E[h(V)]\right|\nonumber\\
&\qquad \le C_1(h)\lambda_2(\phi)\sum_{1\le i\neq j\le N} 
\left(\E[ A_N^0(i,j)^2I(| A_N^0(i,j)|> K)\right.\nonumber\\
&\qquad\qquad +\left.\E[ A_N^g(i,j)^2I(| A_N^g(i,j)|> K)\right)\nonumber\\
&\qquad\qquad + C_2(h)\frac{\lambda_3(\phi)}{ (N\vep_N)^{3/2}}
\sum_{i\neq j}\left(\E[ A_N^0(i,j)^3I(| C_N(i,j)|> k)\right.\nonumber\\
&\qquad\qquad +\left.\E[ (A_N^g(i,j)^3 I(|A_N^g(i,j)|> k)\right)\,.
\label{eq:errorinv}
\end{align}
Using the fact that $\vep_N \downarrow 0$, we have that $\E[ A_N^0(i,j)^4] = \mathrm{O}
(N^{-2}\vep_N^{-1})$. Also
\[
P( |A_N^0(i,j)|>K )\le \mathrm{O}(N^{-1}).
\]
So, by the Cauchy-Schwartz inequality and the above bounds, we have
\[
\E[ A_N^0(i,j)^2I(| A_N^0(i,j)|> K)\le  \mathrm{O}\left(\vep_N^{-1/2} N^{-3/2}\right).
\]
Since $N\vep_N\to \infty$, we have
\[
\lambda_2(\phi)\sum_{1\le i\neq j\le N} \E[ A_N^0(i,j)^2I(| A_N^0(i,j)|> K)
\le C N^{-1/2} \vep_N^{-1/2}\to 0\,,\quad N\to\infty\,.
\]
Similarly, we have
\[
\lambda_3(\phi)\sum_{i\neq j}\E[ A_N^0(i,j)^3I(| A_N^0(i,j)|> K)\le \frac{C}{N^{5/2} 
\vep_N^{3/2}} N^2 \vep_N\to 0\,,\quad N\to\infty\,.
\]
Using Gaussian tail bounds, we can also show that the other two terms in \eqref{eq:errorinv} go 
to $0$, which settles \eqref{l.gauss.claim3}. A similar computation can be done for the imaginary 
part in \eqref{UVdef2}, which proves \eqref{l.gauss.claim4}. The proofs of \eqref{l.gauss.claim1} 
and \eqref{l.gauss.claim2} are analogous (and, in fact, closer to the argument in \cite{chatterjee:2005}).
\end{proof}


\subsection{Leading order variance}

Next, we show that another minor tweak to the entries of $A_N^g$ and $\Delta_N^g$ results in 
a negligible perturbation.

\begin{lemma}\label{l.minor}
Define an $N\times N$ matrix $A_N$ by
\begin{eqnarray}
\label{eq.defANglobal}\bar A_N(i,j)
&=& \sqrt{\frac1Nf\left(\frac iN,\frac jN\right)}\,G_{i\wedge j,i\vee j},\qquad 1\le i,j\le N\,,
\end{eqnarray}
and let 
\[
\bar\Delta_N=\bar A_N-X_N\,,
\]
where $X_N$ is a diagonal matrix of order $N$, defined by
\[
X_N(i,i)=\sum_{1\le k\le N,k\neq i}\bar A_N(i,k),\qquad 1\le i\le N\,.
\]
Then
\begin{eqnarray}
\label{l.minor.claim1}
\lim_{N\to\infty} L\left(\esd(A^g_N),\esd(\bar A_N)\right) &=& 0 \quad \text{in probability}\,,\\
\label{l.minor.claim2}
\lim_{N\to\infty} L\left(\esd(\Delta^g_N),\esd(\bar\Delta_N)\right) &=& 0 \quad \text{in probability}\,.
\end{eqnarray}
\end{lemma}

\begin{proof}
To prove \eqref{l.minor.claim2}, yet another application of Fact \ref{fact:HW} implies that
\begin{eqnarray*}
&&\E\left[L^3\left(\esd(\Delta^g_N),\esd(\bar\Delta_N)\right)\right]\\
&&\quad \le \frac1N\E\left(\Tr\left[\left(\Delta_N^g-\bar\Delta_N\right)^2\right]\right)\\
&&\quad = \frac1N{\sum\sum}_{1\le i\neq j\le N}\Var\left(\bar A_N(i,j)-A_N^g(i,j)\right)\\
&&\qquad+\frac1N\sum_{i=1}^N\Var\left(\sum_{j\in\{1,\ldots,N\}\setminus\{i\}}
\left(\bar A_N(i,j)-A_N^g(i,j)\right)\right)+\frac1{N^2}\sum_{i=1}^Nf\left(\frac iN,\frac iN\right)\\
&&\quad =\frac4{N^2}{\sum\sum}_{1\le i<j\le N}f\left(\frac iN,\frac jN\right)
\left(1-\sqrt{1-\vep_Nf\left(\frac iN,\frac jN\right)}\right)^2\\
&&\quad\qquad +\frac1{N^2}\sum_{i=1}^Nf\left(\frac iN,\frac iN\right)\\
&&\quad \to0\,, \qquad N \to \infty\,,
\end{eqnarray*}
because $f$ is bounded. Thus, \eqref{l.minor.claim2} follows. The proof of \eqref{l.minor.claim1} 
is similar.
\end{proof}


\subsection{Decoupling} 

The (diagonal) entries of $X_N$ are nothing but the row sums of $\bar A_N$. However, the 
correlation between an entry of $\bar A_N$ and that of $X_N$ is small. The following decoupling 
lemma shows that it does not hurt when the entries of $X_N$ are replaced by a mean-zero 
Gaussian random variable of the same variance that is independent of $\bar A_N$. 

\begin{lemma}
\label{l.decouple}
Let $(Z_i\colon\,i \ge 1)$ be a family of i.i.d.\ standard normal random variables, independent 
of $(G_{i,j}\colon\,1\le i \le j)$. Define a diagonal matrix $Y_N$ of order $N$ by
\[
Y_N(i,i)=Z_i\,\sqrt{\frac1N\sum_{1\le j\le N,\,j\neq i}f\left(\frac iN,\frac jN\right)}\,,
\qquad 1\le i\le N\,,
\]
and let
\[
\tilde\Delta_N=\bar A_N+Y_N\,.
\]
Then, for every $k\in\bbn$,
\begin{equation}
\label{l.decouple.claim1}
\lim_{N\to\infty}\frac1N\E\left(\Tr\left[(\tilde\Delta_N)^{2k}-(\bar\Delta_N)^{2k}\right]\right)=0\,,
\end{equation}
and
\begin{equation}
\label{l.decouple.claim2}
\lim_{N\to\infty}\frac1{N^2}\E\left(\Tr^2\left[(\tilde\Delta_N)^{k}\right]
-\Tr^2\left[(\bar\Delta_N)^{k}\right]\right)=0\,.
\end{equation}
\end{lemma}

\begin{proof} 
Without loss of generality we may assume that $f \le1$. For $N \ge 1$, define the $N\times N$ 
matrices $\bar M_N$ and $\tilde M_N$ by
\[
\bar M_N(i,j)=
\begin{cases}
N^{-1/2}G_{i\wedge j,i\vee j},&i\neq j\,,\\
N^{-1/2}G_{i,i}-\sum_{k=1,\,k\neq i}^N\bar M_N(i,k),&i=j\,,
\end{cases}
\]
and 
\[
\tilde M_N(i,j)=
\begin{cases}
\bar M_N(i,j),&i\neq j\,,\\
N^{-1/2}G_{i,i}+Z_i\,\sqrt{\frac {N-1}N},&i=j\,.
\end{cases}
\]
Note that, in the special case where $f$ is identically $1$, $\bar M_N$ and $\tilde M_N$ are identical 
to $\bar\Delta_N$ and $\tilde\Delta_N$, respectively. For $k\in\bbn$ and $\Pi$ a partition of 
$\{1,\ldots,2k\}$, let
\begin{eqnarray}
\label{eq.defpsi}
\Psi(\Pi,N) &=& \Bigl\{i\in\{1,\ldots,N\}^{2k}\colon\,i_u=i_v\\
\nonumber&&\,\,\,\,\,\,\,\,\,\,\iff u,v\text{ belong to the same block of }\Pi\Bigr\}\,.
\end{eqnarray}
For fixed $\Pi$ and $N$, an immediate application of Wick's formula shows that, for all $i,j\in\Psi(\Pi,N)$,
\[
\E\left(\prod_{u=1}^{2k}\bar M_N(i_u,i_{u+1})\right)
= \E\left(\prod_{u=1}^{2k}\bar M_N(j_u,j_{u+1})\right)\,,
\]
with the convention that $i_{2k+1}\equiv i_1$, and
\[
\E\left(\prod_{u=1}^{2k}\tilde M_N(i_u,i_{u+1})\right)
= \E\left(\prod_{u=1}^{2k}\tilde M_N(j_u,j_{u+1})\right)\,,
\]
Therefore, for any $i\in\Psi(\Pi,N)$, we can unambiguously define
\[
\psi(\Pi,N)=\E\left(\prod_{u=1}^{2k}\bar M_N(i_u,i_{u+1})\right)
-\E\left(\prod_{u=1}^{2k}\tilde M_N(i_u,i_{u+1})\right)\,.
\]
As shown in \cite[Lemma 4.12]{bryc:dembo:jiang:2006}, for a fixed $\Pi$,
\begin{equation}
\label{l.decouple.eq0}
\lim_{N\to\infty}N^{-1} |\psi(\Pi,N)|\,\#\Psi(\Pi,N) = 0\,,
\end{equation}
where $\#$ denotes cardinality of a set. 

An immediate observation is that, for all $1\le i,j,i^\prime,j^\prime\le N$,
\[
\Cov\left(\tilde M_N(i,j),\tilde M_N(i^\prime,j^\prime\right)=0
\quad \text{ if } \quad
\left(i\wedge j,i\vee j\right)\neq\left(i^\prime\wedge j^\prime,i^\prime\vee j^\prime\right)\,,
\]
and likewise for $\tilde\Delta_N$. Furthermore,
\[
\Var\left(\tilde M_N(i,j)\right)=\Var\left(\bar M_N(i,j)\right),\qquad 1\le i,j\le N\,,
\]
and likewise for $\tilde\Delta_N$ and $\bar M_N$. For $N\ge1$ and $1\le i,j,i^\prime,j^\prime\le N$, 
define
\[
\eta_N(i,j,i^\prime,j^\prime) =
\begin{cases}
\frac{\Cov\left(\bar\Delta_N(i,j),\bar\Delta_N(i^\prime,j^\prime)\right)}
{\Cov\left(\bar M_N(i,j),\bar M_N(i^\prime,j^\prime)\right)}\,,
&\text{if the denominator is non-zero}\,,\\
0\,,&\text{otherwise}\,.
\end{cases}
\]
It is easy to check that the assumption  $f\le1$ ensures that $|\eta_N(i,j,i^\prime,j^\prime)|\le1$. 
Therefore, for all $N$ and $1\le i,j,i^\prime,j^\prime\le N$,
\begin{eqnarray*}
\label{l.decouple.eq1}
\Cov\left(\bar\Delta_N(i,j),\bar\Delta_N(i^\prime,j^\prime)\right)
&=& \eta_N(i,j,i^\prime,j^\prime)\Cov\left(\bar M_N(i,j),\bar M_N(i^\prime,j^\prime)\right)\,,\\
\label{l.decouple.eq2}
\Cov\left(\tilde\Delta_N(i,j),\tilde\Delta_N(i^\prime,j^\prime)\right)
&=&\eta_N(i,j,i^\prime,j^\prime)\Cov\left(\tilde M_N(i,j),\tilde M_N(i^\prime,j^\prime)\right)\,.
\end{eqnarray*}

For  fixed $\Pi$, $N$ and $i\in\Psi(\Pi,N)$, by an appeal to Wick's formula the above implies 
that there exists a $\xi(i,N)\in[-1,1]$ such that
\[
\E\left(\prod_{u=1}^{2k}\bar\Delta_N(i_u,i_{u+1})\right)
-\E\left(\prod_{u=1}^{2k}\tilde\Delta_N(i_u,i_{u+1})\right)=\xi(i,N)\psi(\Pi,N)\,,
\]
and therefore, by \eqref{l.decouple.eq0},
\[
\begin{aligned}
&\sum_{i\in\Psi(\Pi,N)}\left|\E\left(\prod_{u=1}^{2k}\bar\Delta_N(i_u,i_{u+1})\right)
-\E\left(\prod_{u=1}^{2k}\tilde\Delta_N(i_u,i_{u+1})\right)\right|\\
&=\sum_{i\in\Psi(\Pi,N)}|\xi(i,N)|\left|\psi(\Pi,N)\right|
\le |\psi(\Pi,N)|\,\#\Psi(\Pi,N) = o(N)\,, \quad N \to \infty\,.
\end{aligned}
\]
Since this holds for every partition $\Pi$ of $\{1,\ldots,2k\}$, \eqref{l.decouple.claim1} follows. 
The proof of \eqref{l.decouple.claim2} follows along similar lines.
\end{proof}


\subsection{Combinatorics from free probability}

The final preparation is a general result from random matrix theory. To state this, the following 
notions from the theory of free probability are borrowed, the details of which can be found in 
\cite{nica:speicher:2006}. 

\begin{defn}
For an even positive integer $k$, $NC_2(k)$ is the set of non-crossing pair partitions of 
$\{1,\ldots,k\}$. For $\sigma\in NC_2(k)$, its \emph{Kreweras complement} $K(\sigma)$ 
is the maximal non-crossing partition $\bar\sigma$ of $\{\bar 1,\ldots,\bar k\}$, such that 
$\sigma\cup\bar\sigma$ is a non-crossing partition of $\{1,\bar 1,\ldots,k,\bar k\}$. For example,
\begin{eqnarray*}
K\left(\{(1,4),(2,3)\}\right)&=&\{(1,3),(2),(4)\}\,,\\
K\left(\{(1,2),(3,4),(5,6)\}\right)&=&\{(1),(2,4,6),(3),(5)\}\,.
\end{eqnarray*}
The second example is illustrated as: 

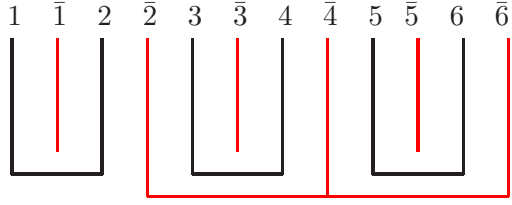
\begin{equation*}
 \setlength{\unitlength}{0.6cm} 
\begin{picture}(9,4)\thicklines
 \put(0,0){\line(0,1){3}} 
\put(0,0){\line(1,0){2}} 
\put(6,0){\line(0,1){3}} 
\put(2,0){\line(0,1){3}}  
\put(4,0){\line(1,0){2}} 
\put(4,0){\line(0,1){3}} 
\put(8,0){\line(0,1){3}} 
\put(8,0){\line(1,0){2}} 
\put(10,0){\line(0,1){3}} 
\put(1,0.5){\color{red}\line(0,1){2.5}}
\put(5,0.5){\color{red}\line(0,1){2.5}}
\put(3,-0.5){\color{red}\line(0,1){3.5}}
\put(7,-0.5){\color{red}\line(0,1){3.5}}
\put(3,-.5){\color{red}\line(1,0){4}}
\put(7,-.5){\color{red}\line(1,0){4}}
\put(11,-.5){\color{red}\line(0,1){3.5}}
\put(9,.5){\color{red}\line(0,1){2.5}}
 \put(-0.1,3.3){$1$} 
\put(0.9,3.3){$\bar{1}$} 
\put(1.9,3.3){$2$} 
\put(2.9,3.3){$\bar{2}$} 
\put(3.9,3.3){$3$}
\put(4.9,3.3){$\bar{3}$}
 \put(5.9,3.3){$4$}
 \put(6.9,3.3){$\bar{4}$} 
\put(7.9,3.3){$5$} 
\put(8.7,3.3){$\bar{5}$}
\put(9.7,3.3){$6$}
\put(10.7,3.3){$\bar{6}$}
\end{picture}
\end{equation*}

\vspace{1cm}

\noindent
For $\sigma\in NC_2(k)$ and $N\ge1$, define
\[
\begin{aligned}
&S(\sigma,N)\\
&=\big\{i\in\{1,\ldots,N\}^k\colon\, i_u=i_v\iff u,v\text{ belong to the same block of }K(\sigma)\big\}
\end{aligned}
\]
and
\[
C(k,N)=\{1,\ldots,N\}^k\setminus\left(\bigcup_{\sigma\in NC_2(k)}S(\sigma,N)\right)\,.
\]
In other words, $S(\sigma,N)$ is the same as $\Psi(K(\sigma),N)$ defined in \eqref{eq.defpsi}. 
\end{defn}

\begin{lemma}
\label{l.diag}
Suppose that, for each $N\ge1$, $W_{N,1},\ldots,W_{N,k}$ are $N\times N$ real (and 
possibly asymmetric) random matrices, where $k$ is a positive even number. Suppose
further that, for each $u=1,\ldots,k$,
\begin{equation}
\label{l.diag.eq0}
\max_{1\le i,j\le N}\E\left[W_{N,u}(i,j)^k\right]=O\left(N^{-k/2}\right)
\end{equation}
and
\begin{equation}
\label{l.diag.eq3}
\lim_{N\to\infty}\E\left[\left(\frac1N\sum_{i\in C(k,N)}P_i\right)^2\,\right]=0\,,
\end{equation}
and that, for every $\sigma\in NC_2(k)$, there exists a deterministic and finite $\beta(\sigma)$ 
such that
\begin{eqnarray}
\label{l.diag.eq1}
\lim_{N\to\infty}\E\left(\frac1N\sum_{i\in S(\sigma,N)}P_i\right)
&=&\beta(\sigma)\,,\\
\label{l.diag.eq2}
\lim_{N\to\infty}\E\left[\left(\frac1N\sum_{i\in S(\sigma,N)}P_i\right)^2\,\right]
&=&\beta(\sigma)^2\,,
\end{eqnarray}
where
\[
P_i=W_{N,1}(i_1,i_2)\ldots W_{N,k-1}(i_{k-1},i_k)W_{N,k}(i_k,i_1)\,, \qquad i\in\{1,\ldots,N\}^k\,.
\]
Furthermore, let $V_1,V_2,\ldots$ be i.i.d.\ random variables drawn from some distribution with 
all moments finite, independent of $(W_{N,j}\colon\,N\ge1,\,1\le j\le k)$, and let
\[
U_N=\diag(V_1,\ldots,V_N), \qquad N \ge1\,.
\]
Then, for all choices of $n_1,\ldots,n_k \ge 0$,
\[
\lim_{N\to\infty} \frac1N\Tr\left(U_N^{n_1}W_{N,1}\ldots U_N^{n_k}W_{N,k}\right) = c 
\quad \text{ in } L^2
\]
for some deterministic $c \in \R$.
\end{lemma}

\begin{proof}
The fact that the sets $S(\sigma,N)$ are disjoint for different $\sigma\in NC_2(k)$ allows us to 
write
\begin{eqnarray*}
\Tr\left(U_N^{n_1}W_{N,1}\ldots U_N^{n_k}W_{N,k}\right)
&=&\sum_{\sigma\in NC_2(k)}\sum_{i\in S(\sigma,N)}\tilde P_i+\sum_{i\in C(k,N)}\tilde P_i\,,
\end{eqnarray*}
where
\[
\tilde P_i=\prod_{j=1}^k\left(V_{i_{j}}^{n_j}W_{N,j}(i_j,i_{j+1})\right),\qquad i\in\{1,\ldots,N\}^k\,.
\]
In order to show that the second sum in the right-hand side is negligible after scaling by $N$, 
the independence of $(V_1,V_2,\ldots)$ and $(W_{N,j}\colon\,N\ge1,\,1\le j\le k)$, together with 
the fact that the common distribution of the former has finite moments, implies that
\begin{eqnarray*}
\E\left[\left(\frac1N\sum_{i\in C(k,N)}\tilde P_i\right)^2\,\right]
\le KN^{-2}\sum_{i,j\in C(k,N)}\E(P_iP_j)\,,
\end{eqnarray*}
where $K$ is a finite constant. Assumption \eqref{l.diag.eq3} shows that
\[
\lim_{N\to\infty} \frac1N\sum_{i\in C(k,N)}\tilde P_i = 0\quad \text{ in } L^2\,.
\]

In order to complete the proof, it suffices to show that for every $\sigma\in NC_2(k)$ 
there exists a $\theta(\sigma)\in\bbr$ with
\begin{equation}
\label{l.diag.eq4}
\lim_{N\to\infty} \frac1N \sum_{i\in S(\sigma,N)}\tilde P_i = \theta(\sigma) \
\quad \text{ in } L^2\,.
\end{equation}
To that end, fix $\sigma\in NC_2(k)$ and note that, for $i\in S(\sigma,N)$,
\begin{equation}
\E(\tilde P_i)
= \E\left(P_i\right)\E\left(\prod_{j=1}^kV_{i_j}^{n_j}\right)
= \E\left(P_i\right)\prod_{u\in K(\sigma)}\E\left(V_1^{\sum_{j\in u}n_j}\right)\,,
\end{equation}
the  product in the last line being taken over every block $u$ of $K(\sigma)$. Putting
\[
\theta(\sigma)=\beta(\sigma)\prod_{u\in K(\sigma)}\E\left(V_1^{\sum_{j\in u}n_j}\right)\,,
\]
we see that \eqref{l.diag.eq1} gives
\begin{equation}
\label{l.diag.eq5}
\lim_{N\to\infty}\E\left[\frac1N\sum_{i\in S(\sigma,N)}\tilde P_i\right]=\theta(\sigma)\,.
\end{equation}
Let us call $i,j\in\bbn^k$ ``disjoint'' if no coordinate of $i$ matches any coordinate of 
$j$, i.e.,
\[
\min_{1\le u,v\le k}|i_u-j_v|\ge1\,.
\]
Since $K(\sigma)$ has exactly $\tfrac12 k+1$ blocks, \eqref{l.diag.eq0} implies that  
\[
\lim_{N\to\infty}N^{-2} \sum_{ \genfrac{}{}{0pt}1{i,j\in S(\sigma,N)}{i,j\text{ not disjoint}}}
\E(\tilde P_i\tilde P_j)=0\,.
\]
If $i,j\in S(\sigma,N)$ are disjoint, then it is immediate that
\[
\E(\tilde P_i\tilde P_j)=\left[\prod_{u\in K(\sigma)}
\E\left(V_1^{\sum_{j\in u}n_j}\right)\right]^2\E(P_iP_j)\,.
\]
The above two displays, in conjunction with \eqref{l.diag.eq2}, show that
\[
\lim_{N\to\infty}\E\left[\left(\frac1N\sum_{i\in S(\sigma,N)}\tilde P_i\right)^2\,\right]
=\theta(\sigma)^2\,.
\]
This, along with \eqref{l.diag.eq5}, establishes \eqref{l.diag.eq4}, from which the proof 
follows.
\end{proof}


\section{Proof of Theorems \ref{t.adj.gen}--\ref{t.mult}}
\label{sec:proofs}

\begin{proof}[Proof of Theorem \ref{t.adj.gen}]
Theorem 2.1 of \cite{chakrabarty:2017} implies that as $N\to\infty$,
\[
\lim_{N\to\infty} \esd\left(\bar A_N\right) = \mu \quad \text{weakly in probability}\,,
\]
for a compactly supported symmetric probability measure $\mu$. Lemma \ref{l.minor} immediately 
tells us that
\[
\lim_{N\to\infty} \esd\left(A_N^g\right) = \mu \quad \text{weakly in probability}\,,
\]
and hence for $h$ and $H_N$ as in Lemma \ref{l.gauss}, 
\[
\lim_{N\to\infty} \E\left[h\left(\Re H_N(A_N^g)\right)\right]=h\left(\Re\int_\bbr\frac1{x-z}\,\mu(dx)\right)\,.
\]
The claim in \eqref{l.gauss.claim1} shows that $A_N^g$ can be replaced by $A_N^0$ in the above 
display. Since the right-hand side is deterministic and the above holds for any $h$ satisfying the 
hypothesis of Lemma \ref{l.gauss}, it follows that
\[
\lim_{N\to\infty} \Re H_N(A_N^0) = \Re\int_\bbr\frac1{x-z}\,\mu(dx) \quad \text{in probability}\,.
\]
A similar argument works for the imaginary part, which shows that
\[
\lim_{N\to\infty} \esd(A_N^0) = \mu \quad \text{weakly in probability}\,.
\]
Lemma \ref{l.mc} completes the proof of \eqref{t.adj.gen.claim1}.

Finally, if $f$ is bounded away from $0$, then the combination of \cite[Lemma 3.1]{chakrabarty:2017} 
and \cite[Corollary 2]{biane:1997} implies that $\mu$ is absolutely continuous with respect to the 
Lebesgue measure (see also \cite{chakrabarty:hazra:2016}). Thus, the proof of Theorem~\ref{t.adj.gen} 
follows. 
\end{proof}

\begin{remark}
{\rm A close inspection of the proof reveals that it suffices to assume that $f$ is bounded and Riemann 
integrable instead of continuous. In other words, if $f$ is symmetric and bounded, and its set of 
discontinuities has Lebesgue measure zero, then the result holds. However, continuity will be used 
later in \eqref{t.lap.gen.eq2} in the proof of Theorem \ref{t.lap.gen}. Furthermore, if $\vep_N=1$ for 
all $N$, then 
\[
\lim_{N\to\infty}\esd\left(N^{-1/2}(A_N-\E(A_N))\right)=\mu_{\sqrt{f(1-f)}} \quad \text{ weakly in probability}\,,
\]
where the right-hand side is the probability measure obtained after replacing $f$ with $\sqrt{f(1-f)}$ 
in \cite[Theorem 2.1]{chakrabarty:2017}.}
\end{remark}

\begin{proof}[Proof of Theorem \ref{t.lap.gen}]
The proof comes in 3 Steps.

\medskip\noindent
{\bf 1.\ Riemann approximation.}
For $N\ge1$, define the $N\times N$ diagonal matrix $Q_N$ by
\begin{equation}
\label{eq.defQNglobal}
Q_N(i,i)=F(i/N)Z_i, \qquad 1\le i\le N\,,
\end{equation}
where
\begin{equation}
\label{eq.defFglobal}
F(x)=\left(\int_0^1f\left(x,y\right)\,dy\right)^{1/2},\qquad 0\le x\le1\,,
\end{equation}
and $(Z_i\colon\,i\ge1)$ is as in Lemma \ref{l.decouple}. Fact \ref{fact0} in Appendix~\ref{sec:facts} 
implies that 
\begin{eqnarray*}
&&\left|\left(\frac1N\Tr\left((\tilde\Delta_N)^k\right)\right)^{1/k}
-\left(\frac1N\Tr\left((\bar A_N+Q_N)^k\right)\right)^{1/k}\right|\\
&&\qquad \le\left(\frac1N\Tr\left[(Y_N-Q_N)^k\right]\right)^{1/k}\,.
\end{eqnarray*}
Since, $f$ being continuous,
\begin{equation}
\label{t.lap.gen.eq2}
\begin{aligned}
&\E\left[N^{-2}\Tr^2\left[(Y_N-Q_N)^k\right]\right] =O(1)\\
&\times\sup_{0\le x\le1}\left[F(x)-\left(\frac1N\sum_{j=1,j\neq[Nx]/N}^N
f\left(x,\frac jN\right)\right)^{1/2}\right]^{2k} \to 0\,, \quad N \to \infty\,,
\end{aligned}
\end{equation}
we get, for every even $k$, 
\begin{equation}
\label{t.lap.gen.eq1}\left(\frac1N\Tr\left((\tilde\Delta_N)^k\right)\right)^{1/k}
-\left(\frac1N\Tr\left((\bar A_N+Q_N)^k\right)\right)^{1/k} \to 0 
\, \text{ in } L^{2k}\,, \quad N\to\infty\,. 
\end{equation}

Our next step is to show that, for every even integer $k$,
\begin{equation}
\label{t.lap.gen.eq3}
\lim_{N\to\infty} \frac1N\Tr\left((\bar A_N+Q_N)^k\right) = \gamma_k \quad \text{ in } L^2 
\end{equation}
for some $\gamma_k\in\bbr$. The above will follow once we show that, for all $m\ge1$ 
and $n_1,\ldots,n_m\ge0$,
\begin{equation}
\label{t.lap.gen.eq4}
\lim_{N\to\infty} \frac1N\Tr\left(Q_N^{n_1}\bar A_N\ldots Q_N^{n_m}\bar A_N\right) 
= \theta \quad \text{ in } L^2
\end{equation}
for some $\theta\in\bbr$ (depending on $m,n_1,\ldots,n_m$). To that end, define the diagonal 
matrices $U_N$ and $B_N$ by
\begin{eqnarray*}
U_N(i,i)&=&Z_i\,,\\
B_N(i,i)&=&F(i/N)\,, \qquad i=1,\ldots,N\,.
\end{eqnarray*}
Observe that 
\[
Q_N=B_NU_N=U_NB_N\,,
\]
and hence the left-hand side of \eqref{t.lap.gen.eq4} is the same as
\begin{equation}
\label{t.lap.gen.eq5}\frac1N\Tr\left(U_N^{n_1}W_{N1}\ldots U_N^{n_m}W_{Nm}\right)\,,
\end{equation}
where 
\[
W_{Nj}=B_N^{n_j}\bar A_N,\qquad j=1,\ldots,m\,.
\]
In order to apply Lemma~\ref{l.diag} we need to verify its hypotheses.

\medskip\noindent
{\bf 2.\ Verification of the hypotheses.}
Our next claim is that $W_{N1},\ldots,W_{Nm}$ satisfy \eqref{l.diag.eq0}--\eqref{l.diag.eq2}. 
To that end, observe that for $N\ge1$ and  $j=1,\ldots,m$,
\[
W_{Nj}(u,v)=F^{n_j}\left(\frac uN\right)f^{1/2}\left(\frac uN,\frac vN\right)N^{-1/2}
G_{u\wedge v,\,u\vee v},\qquad 1\le u,v\le N\,.
\]
Let
\[
H_j(x,y)=F^{n_j}(x)f^{1/2}(x,y),\,(x,y)\in[0,1]^2\,.
\]
Fix a partition $\Pi$ of $\{1,\ldots,m\}$. Recall the notation $\Psi(\Pi,N)$ in the proof of Lemma 
\ref{l.decouple}. Clearly, for every $i\in\Psi(\Pi,N)$,
\[
\E\left(\prod_{j=1}^mW_{Nj}(i_j,i_{j+1}))\right)=N^{-m/2}\psi(\Pi)
\left(\prod_{j=1}^mH_j\left(\frac{i_j}N,\frac{i_{j+1}}N\right)\right)\,,
\]
where
\[
\psi(\Pi)=\E\left(\prod_{j=1}^mG_{i_j\wedge i_{j+1},i_j\vee i_{j+1}}\right)\,,
\]
which does not depend on $i\in\Psi(\Pi,N)$. The standard arguments leading to a proof via the 
method of moments of the Wigner semicircle law show that
\begin{eqnarray*}
&&\lim_{N\to\infty} N^{-m/2+1}\,\psi(\Pi)\,\#\Psi(\Pi,N)\\
&&= \begin{cases}
1,&\text{if }m\text{ is even, and }\Pi=K(\sigma)\text{  for some }\sigma\in NC_2(m)\,,\\
0,&\text{ otherwise}\,.
\end{cases}
\end{eqnarray*}
Assume for the moment that $m$ is even, and let $\sigma\in NC_2(m)$. It is known that $K(\sigma)$ 
has $m/2+1$ blocks. Define a function ${\mathcal L}_\sigma:\{1,\ldots,m\}\to\{1,\ldots,\tfrac12 m+1\}$ 
such that
\[
{\mathcal L}_\sigma(j)={\mathcal L}_\sigma(k)\text{ if and only if }
j,k\text{ are in the same block of }K(\sigma)\,.
\]
It follows that for $\Pi=K(\sigma)$,
\begin{eqnarray*}
&&\lim_{N\to\infty}\frac1N\sum_{i\in\Psi(\Pi,N)}\E\left(\prod_{j=1}^mW_{Nj}(i_j,i_{j+1}))\right)\\
&&\qquad =\int_{[0,1]^{(m/2)+1}}\prod_{(u,v)\in\sigma,u<v}H_u
\left(x_{{\mathcal L}_\sigma(u)},x_{{\mathcal L}_\sigma(v)}\right)dx_1\ldots dx_{(m/2)+1}\,.
\end{eqnarray*}
This shows that hypothesis \eqref{l.diag.eq1} holds. The hypotheses \eqref{l.diag.eq3} and 
\eqref{l.diag.eq2} follow similarly by an analogue of the standard arguments, while \eqref{l.diag.eq0} 
is trivial.

Thus, $W_{N1},\ldots,W_{Nm}$ and $U_N$ satisfy the hypotheses of Lemma \ref{l.diag}. The 
claim of that lemma shows that the random variable in \eqref{t.lap.gen.eq5} converges in $L^2$ 
to a finite deterministic constant as $N\to\infty$, i.e., \eqref{t.lap.gen.eq4} holds. This in turn 
proves \eqref{t.lap.gen.eq3}, which in conjunction with \eqref{t.lap.gen.eq1} shows that
\[
\lim_{N\to\infty} \frac1N\Tr\left((\tilde\Delta_N)^k\right) = \gamma_k \quad \text{ in } L^2\,.
\]
Lemma \ref{l.decouple} asserts that 
\begin{equation}
\label{t.lap.gen.eq6}
\lim_{N\to\infty} \frac1N\Tr\left((\bar\Delta_N)^k\right) = \gamma_k \quad \text{ in } L^2\,,
\end{equation}
and hence also in probability. 

\medskip\noindent
{\bf 3.\ Uniqueness of the limiting measure.} 
Equation \eqref{t.lap.gen.eq3} ensures that there exists a symmetric probability measure 
on $\bbr$ whose $k$-th moment is $\gamma_k$ for every even integer $k$. Our next claim 
is that such a measure is unique, i.e., $(\gamma_k\colon\,k\ge 1)$ determines the measure. 
It is not obvious how to check Carleman's condition, and therefore we argue as follows.   
It suffices to exhibit a probability measure $\nu$ whose odd moments are zero and whose 
$k$-th moment is $\gamma_k$ for even $k$ such that
\begin{equation}
\label{t.lap.gen.eq7}
\int_{-\infty}^\infty e^{tx}\nu(dx)<\infty \quad \forall\,t\in\bbr\,.
\end{equation}
To do so we bring in the notion of a non-commutative probability space (NCP), which is defined 
in Appendix \ref{sec:facts}. For $K>0$ and $N\ge1$, define 
\[
U_{NK}=\diag\left(Z_1\one(|Z_1|\le K),\ldots,Z_1\one(|Z_N|\le K)\right)\,,
\]
and
\[
Q_{NK}=B_NU_{NK}\,.
\]
The arguments leading to \eqref{t.lap.gen.eq4} can be easily tweaked to show that, 
for fixed $K>0$ and a fixed polynomial $p$ in two non-commuting variables,
\begin{equation}
\label{eq.limit}
\lim_{N\to\infty}\frac1N\E\Tr\left[p\left(\bar A_N,Q_{NK}\right)\right]
\end{equation}
exists. Fact \ref{fact:ncp} in Appendix \ref{sec:facts} implies that there exist self-adjoint 
elements $q$ and $a$ in a tracial NCP $(\A,\phi)$ such that the above limit equals 
$\phi\left[p\left(a,q\right)\right]$ for every polynomial $p$ in two non-commuting variables. 
Hence
\begin{equation}
\label{eq.lsd}
\lim_{N\to\infty} \eesd\left[p\left(\bar A_N,Q_{NK}\right)\right] 
= {\mathcal L}\left[p\left(a,q\right)\right] \quad \text {in distibution}\,,
\end{equation}
for any symmetric polynomial $p$, where $\eesd$ denotes the expectation of $\esd$.
Theorem \ref{t.adj.gen} implies that the LSD of $\bar A_N$, which is ${\mathcal L}(a)$ 
by \eqref{eq.lsd}, is compactly supported, and hence $a$ is a bounded element. The 
spectrum of $q$ is clearly a subset of $[-K,K]$. The second claim in Fact \ref{fact:ncp} 
in Appendix \ref{sec:facts} allows us to assume that $(\A,\phi)$ is a $W^*$-probability 
space.

Let
\begin{equation}
\label{eq.defnuk}\nu_K={\mathcal L}(a+q)\,.
\end{equation}
If $C$ is a finite constant such that
\begin{equation}
\label{eq.c}-C\one\le a\le C\one\,,
\end{equation}
then clearly
\begin{equation}
\label{t.lap.gen.eq8}
a+q\le C\one+q\,.
\end{equation}
Applying the method of moments to $Q_{NK}$, we find by an appeal to \eqref{eq.lsd} that the 
law of $q$ is same as the law of 
\begin{equation*}
\label{eq.lawq}F(V)Z_1\one(|Z_1|\le K)\,, 
\end{equation*}
where $V$ is standard uniform independently of $Z_1$, and $F$ is as in \eqref{eq.defFglobal}. 
Under the assumption that $f \le 1$, which represents no loss of generality,
\[
\int_{-\infty}^\infty e^{tx}\left({\mathcal L}(q)\right)(dx)\le e^{t^2/2},\,t\in\bbr\,.
\]
By \cite[Corollary 3.3]{bercovici:voiculescu:1993} applied to \eqref{t.lap.gen.eq8}, it follows that 
\begin{equation}
\label{t.lap.gen.eq9}
\int_\R e^{tx}\nu_K(dx) \le \int_\R e^{tx}\left({\mathcal L}(C\one+q)\right)(dx)
\le \exp\left(\tfrac12 t^2 + tC\right)\,, \quad t >0\,.
\end{equation}

Fact \ref{fact:HW} applied to $\bar A_N+Q_{NK_1}$ and $\bar A_N+Q_{NK_1}$ shows that 
\[
\sup_{N\ge1}L\left(\eesd\left(\bar A_N+Q_{NK_1}\right),
\eesd\left(\bar A_N+Q_{NK_2}\right)\right)
\]
is small for large $K_1$ and $K_2$. Thus, $(\nu_K\colon\,K>0)$ is Cauchy in the L\'evy metric, 
and hence there exists a probability measure $\nu$ such that
\begin{equation}
\label{eq.wc}\lim_{K\to\infty} \nu_K = \nu\,.
\end{equation}
This, along with \eqref{t.lap.gen.eq9}, establishes that
\begin{equation}
\label{t.lap.gen.eq10}
\int_\R e^{tx}\nu(dx) \le \exp\left(\tfrac12 t^2 + tC\right),\quad t>0\,,
\end{equation}
and
\[
\lim_{K\to\infty} \int_\R x^k\nu_K(dx) = \int_\R x^k\nu(dx),\quad k\ge1\,.
\]
Clearly,
\[
\int_{-\infty}^\infty x^k\nu_K(dx)=\lim_{N\to\infty}N^{-1}
\E\Tr\left[\left(\bar A_N+Q_{NK}\right)^k\right]\,.
\]
Therefore, by keeping track of the limit in \eqref{eq.limit}, we can show (with some effort) that
\[
\lim_{K\to\infty} \int_\R x^k\nu_K(dx) =
\begin{cases}
\gamma_k,&k\text{ even}\,,\\
0,&k\text{ odd}\,.
\end{cases}
\]
Thus, $\nu$ has the desired moments. By extending \eqref{t.lap.gen.eq10} to the case $t<0$,  
we see that \eqref{t.lap.gen.eq7} follows. Thus, $\nu$ is the only symmetric probability measure 
whose even moments are $(\gamma_k)$.

Equation \eqref{t.lap.gen.eq6} and the claim proved above show that
\[
\lim_{N\to\infty} \esd\left(\bar\Delta_N\right)= \nu \quad \text{ weakly in probability}\,.
\]
Hence Lemmas \ref{l.mc}--\ref{l.minor} imply that
\[
\lim_{N\to\infty} \esd\bigl((N\vep_N)^{-1/2}(\Delta_N-D_N)\bigr) = \nu
\quad \text{ weakly in probability}\,.
\]
as in the proof of Theorem \ref{t.adj.gen}.

It remains to show that if $f$ is not identically zero, then the support of $\nu$ is unbounded. To that 
end, recall that \eqref{t.lap.gen.eq3}, together with the fact that $\nu$ is the only symmetric probability 
measure whose even moments are $(\gamma_k)$, establish that
\[
\lim_{N\to\infty}\esd\left(\bar A_N+Q_N\right)=\nu \quad \text{ weakly in probability}\,,
\]
where $\bar A_N$ and $Q_N$ are as in \eqref{eq.defANglobal} and \eqref{eq.defQNglobal}, respectively.
Fix  $0<p<1/2$, and for any $N\times N$ real symmetric matrix $\Sigma$, enumerate its eigenvalues in 
descending order by $\lambda_1(\Sigma),\ldots,\lambda_N(\Sigma)$. Weyl's inequality (see Eq.(1.54) 
on page 47 of \cite{tao2012topics}) implies that
\[
\lambda_{2[Np]-1}(Q_N)\le\lambda_{[Np]}\left(\bar A_N+Q_N\right)+\lambda_{[Np]}(-\bar A_N)\,,
\]
where $[x]$ is the smallest integer larger than or equal to $x$. Therefore
\begin{align*}
\limsup_{N\to\infty}\lambda_{[Np]}\left(\bar A_N+Q_N\right)
&\ge \limsup_{N\to\infty}\lambda_{2[Np]-1}(Q_N)-\liminf_{N\to\infty}\lambda_{[Np]}(-\bar A_N)\\
&\ge \limsup_{N\to\infty}\lambda_{2[Np]-1}(Q_N)-C\,,
\end{align*}
where $C$ is as in \eqref{eq.c}.  Letting $p\to0$ and appealing to Fact \ref{f:rep}, we find that
\begin{align*}
\sup\left(\supp (\nu)\right)=&\lim_{p \downarrow 0} 
\limsup_{N\to\infty}\lambda_{[Np]}\left(\bar A_N+Q_N\right)\\
&\ge \lim_{p \downarrow 0}\limsup_{N\to\infty}\lambda_{2[Np]-1}(Q_N)-C =\infty\,,
\end{align*}
where the last line uses the fact that, as $N\to\infty$, $\esd(Q_N)$ converges weakly in probability 
to the distribution of $F(V)Z_1$, the support of which is unbounded because $f$ is not identically zero. 
\end{proof}

\begin{proof}[Proof of Theorem \ref{t.mult}]
Let $(G_{i,j}\colon\,1\le i\le j)$ and $(Z_i\colon\,i\ge1)$ be as in Lemma \ref{l.decouple}. For 
$N\ge1$, define the $N\times N$ matrices
\begin{eqnarray*}
G_N(i,j) &=& N^{-1/2}G_{i\wedge j,i\vee j},\quad 1\le i,j\le N\,,\\
R_N &=& \diag\left(\sqrt{r(1/N)},\ldots,\sqrt{r(1)}\right)\,,\\
U_N &=& \diag(Z_1,\ldots,Z_N)\,.
\end{eqnarray*}
The notation $U_N$ is exactly as in the proof of Theorem \ref{t.lap.gen}. Let $\bar A_N$ and 
$Q_N$ be as in \eqref{eq.defANglobal} and \eqref{eq.defQNglobal}, respectively. Observe that, 
under the assumption \eqref{eq.mult},
\[
\bar A_N=R_NG_NR_N\,,
\]
and
\[
Q_N=\alpha R_N^{1/2}U_NR_N^{1/2}\,,
\]
where $\alpha$ is as defined in the statement of Theorem \ref{t.mult}. Proceeding as in the 
proofs of Theorems \ref{t.adj.gen} and \ref{t.lap.gen}, wee see that it suffices to show that 
\begin{equation}
\label{t.mult.eq4}
\lim_{N\to\infty} \esd\left(R_NG_NR_N\right) = {\mathcal L}\left(r^{1/2}(T_u)T_sT^{1/2}(T_u)\right)
\quad \text{ weakly in probability} 
\end{equation}
and
\begin{equation}
\label{t.mult.eq5}
\begin{aligned}
&\lim_{N\to\infty} \esd\left(R_NG_NR_N+\alpha R_N^{1/2}U_NR_N^{1/2}\right)\\
&= {\mathcal L}\left(r^{1/2}(T_u)T_sT^{1/2}(T_u)+\alpha r^{1/4}(T_u)T_gr^{1/4}(T_u)\right)
\text{ weakly in probability}\,, 
\end{aligned}
\end{equation}
where $T_s,T_g,T_u$ are as in the statement. Define $U_{NK}$ to be the ``truncated'' version 
of $U_N$, for a fixed $K>0$, as in the proof of Theorem \ref{t.lap.gen}. Both \eqref{t.mult.eq4} 
and \eqref{t.mult.eq5} will follow once we show that
\begin{equation}
\label{eq.toshow}
\lim_{N\to\infty} \frac1N\Tr\left(p\left(R_N^{1/2},U_{NK},G_N\right)\right)
 = \tau\left(p\left(T_r,T_g^\prime,T_s\right)\right) \quad \text{in probability}\,,
\end{equation}
where $T_r=r^{1/4}(T_u)$ and $T_g^\prime=T_g\one(|T_g|\le K)$, for any symmetric polynomial 
$p$ in three non-commuting variables. It is a well known fact that, for all $k\ge1$,
\begin{equation}
\label{t.mult.eq6}
\lim_{N\to\infty} \frac1N\Tr(G_N^k) = \tau(T_s^k)\quad \text{in probability}\,.
\end{equation}
Since $R_N$ and $U_{NK}$ are diagonal matrices, they commute. This, in conjunction with the 
strong law of large numbers, implies that, for any $k\ge1$, $m_1,\ldots,m_k$ and $n_1,\ldots
,n_k\ge0$,
\[
\begin{aligned}
&\lim_{N\to\infty} \frac1N\Tr\left(R_N^{m_1} U_{NK}^{n_1} \ldots R_N^{m_k}U_{NK}^{n_k}\right)\\
&= \int_0^1du\,r^{(m_1+\ldots+m_k)/4}(u)
\int_{-K}^K(2\pi)^{-1/2} dx\,x^{n_1+\ldots+n_k}e^{-x^2/2}
\quad a.s.
\end{aligned}
\]
The above, in conjunction with \eqref{t.mult.eq3} and the fact that $T_g$ and $T_r$ commute, implies 
that
\begin{equation}
\label{t.mult.eq7}
\lim_{N\to\infty} \frac1N\Tr\left(p\left(R_N^{1/2},U_{NK}\right)\right) 
= \tau\left(p\left(T_r,T_g^\prime\right)\right) \quad a.s.
\end{equation}
for any polynomial $p$ in two variables. 

Thus, all that remains to show is the asymptotic free independence of $T_s$ and $(T_r,T_g^\prime)$, 
which is precisely the claim of Fact \ref{fact:voiculescu} in Appendix \ref{sec:facts}, i.e., \eqref{t.mult.eq6} 
and \eqref{t.mult.eq7} imply \eqref{eq.toshow}. Applying \eqref{eq.toshow} to $p(x,y,z)=x^2zx^2$ and 
$p(x,y,z)=x^2zx^2+\alpha xyx$, we get the truncated versions of \eqref{t.mult.eq4} and \eqref{t.mult.eq5}, 
respectively. Yet another application of Fact \ref{fact:HW} in Appendix \ref{sec:facts} allows us to let 
$K\to\infty$, obtaining \eqref{t.mult.eq4} and \eqref{t.mult.eq5}. This completes the proof of 
\eqref{t.mult.eq1} and \eqref{t.mult.eq2}.
\end{proof}


\section{Proof of Theorem \ref{t.random}}
\label{sec:proofrandom}

\begin{proof}[Proof of Theorem \ref{t.random}]
As before, Lemma \ref{l.mc} and the assumption \eqref{eq.conditions} imply that the mean of the entries of $A_N$ can be subtracted at the cost of a negligible perturbation of the $\esd$. The inequalities \eqref{eq.conditions} and  \eqref{random.eq1} ensure that the Gaussianization as in Lemma~\ref{l.gauss}  goes through by conditioning on $R_{N1},\ldots,R_{NN}$. That is, if $(G_{ij}\colon\,1\le i\le j)$ is a 
collection of i.i.d.\ standard normal random variables that are independent of $(R_{Ni}\colon\,1\le 
i\le N, N\ge1)$,  $W_N^g$ is an $N\times N$ matrix defined by
\[
W_N^g(i,j)=G_{i\wedge j,i\vee j}\,,1\le i,j\le N\,,
\]
and 
\[
\Theta_N=\diag\left(\sqrt{R_{N1}},\ldots,\sqrt{R_{NN}}\right)\,,
\]
then the $\esd$ of $A_N/\sqrt{N\vep_N}$ is close to that of $\Theta_NW_N^g\Theta_N/\sqrt N$. 

The assumptions \eqref{random.eq1} and \eqref{random.eq.weak} imply that 
\[
\lim_{N\to\infty}\frac1N\Tr\left(\Theta_N^{2k}\right)=\int_{-\infty}^\infty x^k\mu_r(dx)\quad\text{ a.s.},\,k\ge1\,.
\]
Finally,  Fact \ref{fact:voiculescu} together with \eqref{random.eq1} shows the asymptotic free independence of $\Theta_N$ and $W_N^g$, that is,
\[
\lim_{N\to\infty}\esd\left(N^{-1/2}\Theta_NW_n^g\Theta_N\right)=\mu_r\boxtimes\mu_s \quad \text{ weakly in probability}\,,
\]
This completes the proof.
\end{proof}


\section{Applications}
\label{sec:appl}

In this section we discuss three applications, explained in Sections~\ref{sec:CRG}--\ref{sec:social}.


\subsection{Constrained random graphs}
\label{sec:CRG}

Let $\mathcal{S}_N$ be the set of all simple graphs on $N$ vertices. Suppose that we fix the 
degrees of the vertices, namely, vertex $i$ has degree $k^*_i$. Here, $k^*=(k^*_i\colon\,
1 \leq i \leq N)$ is a sequence of \emph{positive integers} of which we only require that they 
are \emph{graphical}, i.e., there is at least one simple graph with these degrees. The 
so-called \emph{canonical ensemble} $P_N$ is the unique probability distribution on 
$S_N$ with the following two properties: 
\begin{itemize}
\item[(I)] 
The \emph{average degree} of vertex $i$, defined by $\sum_{G \in \mathcal{S}_N} k_i(G) 
P_N(G)$, equals $k^*_i$ for all {\color{blue}$1 \leq i \leq N$}.          
\item[(II)] 
The \emph{entropy} of $P_N$, defined by $-\sum_{G \in \mathcal{S}_N} P_N(G) \log P_N(G)$, 
is maximal. 
\end{itemize}
The name canonical ensemble comes from Gibbs theory in equilibrium statistical physics. 
The probability distribution $P_N$ describes a random graph of which we have \emph{no 
prior information} other than the average degrees, and is called the \emph{soft configuration 
model}. It is known that, because of property (II), $P_N$ takes the form (\cite{jaynes:1957})
\[
P_N(G) = \frac{1}{Z_N(\theta^*)} \exp\left[-\sum_{i=1}^N \theta^*_i k_i(G)\right]\,, 
\qquad G \in \mathcal{S}_N\,,
\] 
where $\theta^*=(\theta^*_i\colon\,1 \leq i \leq N)$ is a sequence of real-valued Lagrange multipliers 
that must be chosen in such a way that property (I) is satisfied. The normalization constant 
$Z_N(\theta^*)$, which depends on $\theta^*$, is called the partition function in Gibbs theory. 

The matching of property (I) uniquely fixes $\theta^*$, namely, it turns out that 
(\cite{squartini:demol:denhollander:garlaschelli:2015}) 
\[
P_N(G) = \prod_{1\le i< j\le N}^N (p^*_{ij})^{A_N[G](i,j)}\,(1-p^*_{ij})^{1-A_N[G](i,j)}\,, 
\qquad G \in \mathcal{S}_N\,,
\]
where $A_N[G]$ is the adjacency matrix of $G$, and $p^*_{ij}$ represent a \emph{reparameterisation} 
of the Lagrange multipliers, namely,
\begin{equation}
\label{eq.pijpre}
p^*_{ij} = \frac{x^*_i x^*_j}{1+x^*_i x^*_j}\,, \qquad 1 \leq i \neq j \leq N\,,
\end{equation}
with $x^*_i = e^{-\theta^*_i}$. Thus, we see that $P_N$ is nothing other than an inhomogeneous 
Erd\H{o}s-R\'enyi random graph where the probability that vertices $i$ and $j$ are connected by 
an edge equals $p^*_{ij}$. In order to match property (I), these probabilities must satisfy
\begin{equation}
\label{eq.kirels}
k^*_i = \sum_{ j\in\{1,\ldots,N\}\setminus\{i\} } p^*_{ij}\,, \qquad 1 \leq i \leq N\,,
\end{equation}
which constitutes a set of $N$ equations for the $N$ unknowns $x_1^*,\ldots,x_N^*$. 

In order to state the next result, we need to make some assumptions on the sequence 
$(k^*_{Ni}\colon\,1\le i\le N)$. For the sake of notational simplification, the dependence 
on $N$ will be suppressed from the notation.

\begin{prop}
Let $(k^*_i\colon\,1\le i\le N)$ be a graphical sequence of positive integers. Define 
\[
m_N = \max_{1 \leq \ell \leq N} k^*_\ell\,.
\] 
Assume that
\begin{equation}
\label{eq.sparse}
\lim_{N\to\infty} m_N=\infty\,, \qquad \lim_{N\to\infty} m_N/\sqrt{N}=0\,,
\end{equation}
and
\begin{equation}
\label{eq.weak}
\lim_{N\to\infty} \frac1N\sum_{i=1}^N\delta_{k_i^*/m_N}=\mu_r \quad \text{ weakly}\,,
\end{equation}
for some probability measure $\mu_r$. Let $x^*_i$ and $p^*_{ij}$ be determined by \eqref{eq.pijpre} 
and \eqref{eq.kirels}. Let $A_N$ be the adjacency matrix of an inhomogeneous Erd\H{o}s-R\'enyi 
random graph on $N$ vertices, with $p^*_{ij}$ the probability of an edge being present between 
vertices $i$ and $j$ for $1 \le i\neq j \le N$. Then 
\[
\lim_{N\to\infty} \esd\bigl((N\vep_N)^{-1/2}A_N\bigr) = 
\mu_r\boxtimes\mu_s \quad \text{ weakly in probability}\,.
\]
\end{prop}

\begin{proof}
Abbreviate 
\[
\sigma_N = \sum_{1 \leq \ell \leq N} k^*_\ell\,, 
\]
It is known that (\cite{squartini:demol:denhollander:garlaschelli:2015})
\[
\max_{1 \leq \ell \leq N} x^*_\ell = o(1)\,,
\]
in which case \eqref{eq.pijpre} and \eqref{eq.kirels} give
\begin{equation}
\label{eq.xipij}
x^*_i = [1+o(1)]\,\frac{k^*_i}{\sqrt{\sigma_N}}\,,\qquad
p^*_{ij} = [1+o(1)]\,\frac{k^*_i k^*_j}{\sqrt{\sigma_N}}\,, \qquad N\to\infty\,,
\end{equation}
with the error term \emph{uniform} in $1 \le i \neq j \le N$. Pick 
\[
\vep_N = \frac{m_N^2}{\sigma_N} \,.
\] 
It follows from \eqref{eq.sparse} that
\[
\lim_{N\to\infty} \vep_N=0,\,\qquad \lim_{N\to\infty} N\vep_N=\infty\,.
\]
As in the proof of Theorem \ref{t.random}, Lemmas \ref{l.mc}--\ref{l.gauss} imply that the 
upper triangular entries of $A_N$ can be replaced by independent mean-zero normal random 
variables. In other words, if $(G_{ij}\colon\,1\le i\le j)$ are i.i.d.\ standard normal, and $A_N^g$ 
is the random matrix defined by
\[
A_N^g(i,j)=\sqrt{p^*_{ij}}\,G_{i\wedge j,i\vee j}\,,\qquad 1\le i, j\le N\,,
\]
with $p^*_{ii}=0$ for all $1 \le i \le N$, then $\esd\bigl((N\vep_N)^{-1/2}A_N\bigr)$ and 
$\esd\bigl((N\vep_N)^{-1/2}A_N^g\bigr)$ are asymptotically close. The second part of 
\eqref{eq.xipij} implies that
\[
\sqrt{p^*_{ij}}=\bigl[1+o(1)\bigr]\,\sqrt{\vep_N\,\frac{k_i^*k_j^*}{m_N^2}}\,,
\]
uniformly in $1\le i\neq j\le N$, and hence
\[
\sum_{i,j=1}^N\left[\sqrt{p^*_{ij}}-\sqrt{\vep_N\,\frac{k_i^*k_j^*}{m_N^2}}\,\right]^2
=o\left(N^2\vep_N\right)\,.
\]
In other words, if $\tilde A_N$ is defined by
\[
\tilde A_N(i,j)=\sqrt{\frac{k_i^*k_j^*}{m_N^2}}\,G_{i\wedge j,i\vee j}\,, \qquad 1\le i,j\le N\,,
\]
then 
\[
\lim_{N\to\infty}\frac1N\E\left[\Tr\bigl((N\vep_N)^{-1/2}A_N^g-N^{-1/2}\tilde A_N\bigr)^2\right]=0\,.
\]
Fact \ref{fact:HW} implies that 
\[
\lim_{N\to\infty}L\left(\esd\bigl((N\vep_N)^{-1/2}A_N^g\bigr),
\esd\bigl(N^{-1/2}\tilde A_N\bigr)\right)=0 \quad \text{ in probability}\,.
\]
Finally, by an appeal to Fact \ref{fact:voiculescu}, \eqref{eq.weak} implies that
\[
\lim_{N\to\infty}\esd\bigl(N^{-1/2}\tilde A_N\bigr)=\mu_r\boxtimes\mu_s \quad \text{ weakly in probability}\,,
\]
where $\mu_s$ is the standard semicircle law. Hence
\[
\lim_{N\to\infty}\esd\bigl((N\vep_N)^{-1/2}A_N\bigr)=\mu_r\boxtimes\mu_s \quad \text{ weakly in probability}\,.
\]
\end{proof}

\begin{remark} 
{\rm We look at a concrete example of a graphical sequence $(k^*_i\colon\,1\le i\le N)$ 
satisfying \eqref{eq.sparse}--\eqref{eq.weak}. For $N\ge1$, let
\[
k^*_i=\lfloor i^{1/3}\rfloor\,, \qquad 1\le i\le N\,.
\]
Then \cite[Theorem 7.12]{vanderhofstad:book} implies that $(k^*_i\colon\,1\le i\le N)$ is 
graphical for $N$ large enough. Since $m_N=\lfloor N^{1/3}\rfloor$, it is immediate that
\eqref{eq.sparse} holds and that
\[
\lim_{N\to\infty} \left(\frac1N\sum_{i=1}^N\delta_{k^*_i/m_N}\right)(\cdot)
=P(U^{1/3}\in\cdot\,) \quad \text{ weakly}
\]
with $U$ a standard uniform random variable.}
\end{remark}


\subsection{Chung-Lu graphs}
\label{sec:Chung-Lu} 

The following random graph introduced by \cite{chung:lu:2002} is similar to the one discussed 
in Section~\ref{sec:CRG}. For $N\ge1$, let $(d_{Ni}\colon\,1\le i\le N)$ be a sequence of positive 
real numbers. Abbreviate
\[
m_N=\max_{1\le i\le N}d_{Ni}\,, \qquad \sigma_N=\sum_{i=1}^Nd_{Ni}\,.
\]
Assume that
\[
\lim_{N\to\infty} \frac{m_N^2}{\sigma_N} = 0\,,\qquad \lim_{N\to\infty} N\frac{m_N^2}{\sigma_N} = \infty\,,
\]
and
\[
\lim_{N\to\infty}\frac1N\sum_{i=1}^N\delta_{d_{Ni}/m_N}=\mu_r \quad \text{ weakly}
\]
for some measure $\mu_r$ on $\bbr$. Consider an inhomogeneous Erd\H{o}s-R\'enyi graph on $N$ 
vertices where an edge exists between $i$ and $j$ with probability $d_{Ni}d_{Nj}/\sigma_N$, for 
$1\le i\neq j\le N$, which is called a Chung-Lu graph. If $A_N$ denotes its adjacency matrix, then 
the following result is a corollary of Theorem \ref{t.random}.

\begin{prop}
Under the hypotheses mentioned above,
\[
\lim_{N\to\infty}\esd\left((N\vep_N)^{-1/2}A_N\right)
= \mu_r\boxtimes\mu_s \quad \text{ weakly in probability}\,,
\]
where
\[
\vep_N=\frac{m_N^2}{\sigma_N}
\]
and $\mu_s$ is the standard semicircle law.
\end{prop}


\subsection{Social networks}
\label{sec:social}

Consider a community consisting of $N$ individuals. Data is available on whether the $i$-th individual 
and the $j$-th individual are acquainted, for every pair $\{i,j\}$ with $1 \leq  i,j \leq N$. Based on this 
data, the \emph{sociability pattern} of the community has to be inferred statistically. Examples arise 
in social networks and collaboration networks.

The above situation can be modeled in several ways, one being the following. Denote by $\rho$ the 
\emph{sociability distribution} of the community, which is a compactly supported probability measure 
on $[0,\infty)$. Let $(R_i)_{1 \leq i \leq N}$ be i.i.d.\ random variables drawn from $\rho$. Think of 
$R_i$ as the \emph{sociability index} of the $i$-th individual. Fix $\vep_N>0$ such that $\vep_N m^2 
\leq 1$, where $m$ is the supremum of the support of $\rho$, so that
\begin{equation}
0 \leq \vep_N R_i R_j \leq 1\,, \qquad 1 \leq i \neq j \leq N\,.
\end{equation}
Suppose that, conditional on $(R_i)_{1 \leq i \leq N}$, the $i$-th and the $j$-th individual are acquainted 
with probability $\vep_NR_iR_j$. In other words, the graph in which the vertices are individuals and the 
edges are mutual acquaintances is an inhomogeneous Erd\H{o}s-R\'enyi random graph $\bbg_N$ with 
random connection parameters that are controlled by $\nu$. The data that is available is the adjacency 
matrix $A_N$ of this graph. The goal is to draw information about $\rho$ from this data. This statistical 
inference problem boils down to estimating $\rho$ from $A_N$.  Without loss of generality we assume 
that $\rho$ is standardized, i.e.,
\begin{equation}
\label{eq.mean}
\int_0^\infty x\rho(dx)=1\,.
\end{equation}

\begin{prop}\label{prop.stat}
Under the assumptions $N^{-1}\ll\vep_N\ll1$ and \eqref{eq.mean},
\[
\lim_{N\to\infty}\esd\left(\sqrt{\frac{ N}{\Tr(A_N^2)}}\,A_N\right)
=\rho\boxtimes\mu_s \quad \text{ weakly in probability}\,,
\]
where $\mu_s$ is the standard semicircle law.
\end{prop}

\begin{proof}
It is immediate that
\[
\lim_{N\to\infty}\frac1N\sum_{i=1}^N\delta_{R_i}=\rho \quad \text{ weakly, almost surely}\,.
\]
Theorem \ref{t.random} implies that if $N^{-1}\ll\vep_N\ll1$, then
\begin{equation}
\label{ps.eq0}
\lim_{N\to\infty}\esd\left((N\vep_N)^{-1/2}A_N\right)=\rho\boxtimes\mu_s \quad \text{ weakly in probability}\,.
\end{equation}
Since $A_N(i,j)$ is either $0$ or $1$,
\[
\E\left(\Tr(A_N^2)\right)=\sum_{i=1}^N\sum_{j=1}^N\E\left[A_N(i,j)\right]
=\sum_{1\le i\neq j\le N}\vep_N\E(R_iR_j)=\vep_NN(N-1)\,,
\]
where the last equality follows from \eqref{eq.mean}. Consequently,
\begin{equation}
\label{ps.eq1}
\lim_{N\to\infty}\frac1{N^2\vep_N}\E\left(\Tr(A_N^2)\right)=1\,.
\end{equation}
The fact that the variance equals the sum of the expectation of the conditional variance and 
the variance of the conditional expectation, implies that
\begin{align*}
\Var\left(\Tr(A_N^2)\right)
&=\Var\left(2\sum_{1\le i<j\le N}A_N(i,j)\right)\\
&=4E\left(\sum_{1\le i<j\le N}\vep_NR_iR_j(1-\vep_NR_iR_j)\right)
+4\Var\left(\sum_{1\le i<j\le N}\vep_NR_iR_j\right)\\
&=O(N^2\vep_N)+4\vep_N^2\sum_{1\le i<j\le N}\sum_{1\le k<l\le N}\Cov(R_iR_j,R_kR_l)\\
&=O(N^3\vep_N^2)\,,
\end{align*}
where the last line follows from the observation that if $i,j,k,l$ are distinct, then $\Cov(R_iR_j,R_kR_l)$ 
vanishes. Hence,
\[
\lim_{N\to\infty}\Var\left(\frac1{N^2\vep_N}\Tr\left(A_N^2\right)\right)=0\,.
\]
The above in conjunction with \eqref{ps.eq1} shows that
\[
\lim_{N\to\infty}\frac1{N^2\vep_N}\Tr\left(A_N^2\right)=1 \quad \text{ in probability}\,.
\]
This, together with \eqref{ps.eq0}, completes the proof.
\end{proof}

Thus, $\rho\boxtimes\mu_s$ can in principle be statistically estimated from $A_N$. Subsequently, $\rho$ 
can be computed because the moments of $\rho\boxtimes\mu_s$ are functions of the moments of $\rho$, 
as shown below. Eq.(14.5) on page 228 of \cite{nica:speicher:2006} tells us that for $n\ge1$,
\[
\int_{-\infty}^\infty x^{2n}\,\rho\boxtimes\mu_s(dx)
=\sum_{\sigma\in NC_2(2n)}\prod_{j=1}^{n+1}\int_{-\infty}^\infty x^{l_j(\sigma)}\rho(dx)\,,
\]
where $l_1(\sigma),\ldots,l_{n+1}(\sigma)$ are block sizes of $K(\sigma)$, the Kreweras complement of 
$\sigma$. With the help of the above, the $n$-th moment of $\rho$ can be written in terms of the $2n$-th 
moment of $\rho\boxtimes\mu_s$, and the first $n-1$ moments of $\rho$. Therefore, the moments of 
$\rho$ can be recursively computed from those of $\rho\boxtimes\mu_s$. Since $\rho$ is compactly 
supported, it can be computed from its moments. 


\appendix

\section{Basic facts}
\label{sec:facts}

The following is \cite[Corollary A.41]{bai:silverstein:2010}, and is also a corollary of the 
Hoffman-Wielandt inequality.

\begin{fact}\label{fact:HW}
If $L$ denotes the L\'evy distance between two probability measures, then for $N\times N$ 
symmetric matrices $A$ and $B$,
\[
L^3\left(\esd(A),\esd(B)\right)\le\frac1N\Tr\left[(A-B)^2\right]\,.
\]
\end{fact}

The following is a consequence of the Minkowski and $k$-Hoffman-Wielandt inequalities. The
latter can be found in Exercise 1.3.6 of \cite{tao2012topics}.

\begin{fact}
\label{fact0}
For real symmetric matrices $A$ and $B$ of the same order, and an even positive integer $k$,
\[
\left|\Tr^{1/k}(A^k)-\Tr^{1/k}(B^k)\right|\le\Tr^{1/k}\left[(A-B)^k\right]\,.
\]
\end{fact}

\begin{defn}
A \emph{non-commutative probability space (NCP)} $(\A,\phi)$ is a unital\\ $*$-algebra $\A$ 
equipped with a linear functional $\phi\colon\,\A\to\bbc$ that is unital, i.e.,
\[
\phi(\one)=1\,,
\]
and positive, i.e.,
\[
\phi(a^*a)\ge0\text{ for all }a\in\A\,.
\]
An NCP $(\A,\phi)$ is \emph{tracial} if 
\[
\phi(ab)=\phi(ba),\,a,b\in\A\,.
\]
\end{defn}

\begin{fact}
\label{fact:ncp}
Suppose that, for every $n\in\bbn$, $(\A_n,\phi_n)$ is a tracial NCP, and there exist self-adjoint 
$a_{n1},\ldots,a_{nk}\in\A_n$ such that, for every polynomial $p$ in $k$ non-commuting variables,
\begin{equation}
\label{fact:ncp.eq0}
\lim_{n\to\infty}\phi_n\left(p\bigl(a_{n1},\ldots,a_{nk}\bigr)\right)=\alpha_p\in\bbc\,.
\end{equation}
Then there exists a tracial NCP $(\A_\infty,\phi_\infty)$ and self-adjoint $a_{\infty1},\ldots,a_{\infty k}
\in\A_\infty$ such that, for every polynomial $p$ in $k$ non-commuting variables,
\[
\phi_\infty\left(p\bigl(a_{\infty 1},\ldots,a_{\infty k}\bigr)\right)=\alpha_p\,.
\]
Furthermore, if 
\begin{equation}
\label{fact.ncp.eq1}
\sup_{1\le i\le k,\,j\ge1}\left[\phi_\infty\left(a_{\infty i}^{2j}\right)\right]^{1/2j}<\infty\,,
\end{equation}
then $(\A_\infty,\phi_\infty)$ can be embedded into a $W^*$-probability space.
\end{fact}

\begin{proof}
Let
\[
\A_\infty=\bbc[X_1,\ldots,X_k]\,,
\]
the set of all polynomials in $k$ non-commuting variables. For a monomial 
\[
p=\alpha X_{i_1}\ldots X_{i_m}\,,
\]
define
\[
p^*=\overline\alpha X_{i_m}\ldots X_{i_1}\,.
\]
This defines the $*$-operation on the whole of $\A$. Let
\[
\phi_\infty(p)=\alpha_p\text{ for all }p\in\A_\infty\,.
\]
It is immediate from \eqref{fact:ncp.eq0} that $\phi_\infty$ is positive and unital, i.e., 
$(\A_\infty,\phi_\infty)$ is an NCP. The desired conclusions are ensured by defining
\[
a_{\infty 1}=X_1,\,\ldots,\,a_{\infty k}=X_k\,.
\]
Finally, \eqref{fact.ncp.eq1} implies that $a_{\infty,1},\ldots,a_{\infty,k}$ are bounded. 
Hence, by going from polynomials to continuous functions with the help of the 
Bolzano-Weierstrass theorem, we can embed $(\A_\infty,\phi_\infty)$ into a $W^*$-probability 
space. 
\end{proof}

The next fact follows from \cite[Theorem 4.20]{mingo:speicher:2017} (which is due to Voiculescu)
and the discussion immediately following it.

\begin{fact}
\label{fact:voiculescu}
Suppose that $W_N$ is an $N\times N$ scaled standard Gaussian Wigner matrix, i.e., a 
symmetric matrix whose upper triangular entries are i.i.d.\ normal with mean zero and variance 
$1/N$. Let $D_N^1$ and $D_N^2$ be (possibly random) $N\times N$ symmetric matrices 
such that there exists a deterministic $C$ satisfying
\[
\sup_{N\ge1,i=1,2}\|D_N^i\|\le C<\infty\,,
\]
where $\|\cdot\|$ denotes the usual matrix norm (which is same as the largest singular value 
for a symmetric matrix). Furthermore, assume that there is a $W^*$-probability space $(\A,\tau)$ 
in which there are self-adjoint elements $d_1$ and $d_2$ such that, for any polynomial $p$ in 
two variables, it
\[
\lim_{N\to\infty} \frac1N\Tr\left(p\left(D_N^1,D_N^2\right)\right) = \tau\left(p(d_1,d_2)\right) \text{ a.s.}
\]
Finally, suppose that $(D_N^1,D_N^2)$ is independent of $W_N$. Then there exists a self-adjoint 
element $s$ in $\A$ (possibly after expansion) that has the standard semicircle distribution and is 
freely independent of $(d_1,d_2)$, and is such that 
\[
\lim_{N\to\infty} \frac1N\Tr\left(p\left(W_N,D_N^1,D_N^2\right)\right) 
= \tau\left(p(s,d_1,d_2)\right)\text{ a.s.}
\]
for any polynomial $p$ in three variables.
\end{fact}

\begin{fact}\label{f:rep}
Suppose that for all $n\ge1$, $Z_{n1}\ge\ldots\ge Z_{nn}$ are random variables such that
\[
\lim_{n\to\infty}\frac1n\sum_{j=1}^n\delta_{Z_{nj}}=\mu \quad \text{ weakly in probability}\,,
\]
for some probability measure $\mu$ on $\bbr$, where $\delta_x$ is the probability measure 
that puts mass $1$ at $x$. Then,
\[
\lim_{p\to0}\limsup_{n\to\infty}Z_{n\,[np]}=\sup\left(\supp (\mu)\right) \quad \text{ almost surely}\,,
\]
where $[x]$ denotes the smallest integer larger than or equal to $x$. 
\end{fact}

\begin{proof}
Our first claim is that if $x\in\bbr$ and $0<p<1$ are such that 
\[
\mu((-\infty,x))<1-p\,,
\]
then
\begin{equation}
\label{rep.eq1}
\limsup_{n\to\infty}Z_{n\,[np]}\ge x\,,\text{ almost surely}\,.
\end{equation}
To see why, fix $p,x$ as above and $\vep>0$ such that $\mu(\{x-\vep\})=0$, and note that the 
hypothesis implies that 
\begin{equation}
\label{rep.eq2}\lim_{n\to\infty}\frac1n\#\{1\le j\le n:Z_{nj}\le x-\vep\}=\mu((-\infty,x-\vep])
\quad \text{ in probability}\,.
\end{equation}
Therefore, 
\begin{align*}
&P\left(\limsup_{n\to\infty}Z_{n\,[np]}\le x-\vep\right)\\
&\le P\left(\frac1n\#\{1\le j\le n:Z_{nj}\le x-\vep\}\ge1-\frac1n[np]\text{ for large }n\right)\\
&\le\limsup_{n\to\infty}P\left(\frac1n\#\{1\le j\le n:Z_{nj}\le x-\vep\}\ge1-\frac1n[np]\right)=0\,,
\end{align*}
where the last step follows from \eqref{rep.eq2} and the observation that
\[
\lim_{n\to\infty}1-\frac1n[np]=1-p>\mu((-\infty,x))\ge\mu((-\infty,x-\vep])\,.
\]
Since $\vep>0$ can be chosen to be arbitrarily small such that $\mu(\{x-\vep\})=0$, \eqref{rep.eq1} follows.

It is immediate that $\limsup_{n\to\infty}Z_{n\,[np]}$ is monotone in $p$, and hence the a.s.\ limit exists 
as $p \downarrow 0$. Furthermore,
\begin{equation}
\label{rep.eq3}\lim_{p\to0}\limsup_{n\to\infty}Z_{n\,[np]}\le\alpha \quad \text{ a.s.}\,,
\end{equation}
where
\[
\alpha=\sup\left(\supp (\mu)\right)\,.
\]

To complete the proof, choose $x_k$ such that 
\[
x_k\uparrow\alpha\,,
\]
and $x_k<\alpha$. Since $\alpha$ is the right end point of the support of $\mu$, it follows that 
\[
\mu((-\infty,x_k))<1\,.
\]
Choosing 
\[
0<p_k<\left[1-\mu((-\infty,x_k))\right]\wedge\frac1k\,,k\ge1\,,
\]
we see that \eqref{rep.eq1} implies
\[
\limsup_{n\to\infty}Z_{n\,[np_k]}\ge x_k \quad \text{ a.s.}
\]
Therefore,
\[
\liminf_{k\to\infty}\limsup_{n\to\infty}Z_{n\,[np_k]}\ge\alpha \quad \text{ a.s.}\,,
\]
because $x_k\uparrow\alpha$. Since $p_k\downarrow0$, the left-hand side above equals 
that of \eqref{rep.eq3}. 
\end{proof}


\section*{Acknowledgment}
The authors are grateful to an anonymous referee for helpful comments. The research of AC and RSH 
was supported through MATRICS grant of SERB, the research of FdH and MS through NWO Gravitation 
Grant NETWORKS 024.002.003. The authors are grateful to ISI, NETWORKS and STAR for financial 
support of various exchange visits to Delft, Kolkata and Leiden.


\bibliographystyle{abbrvnat}

\end{document}